\documentclass{article}
\usepackage[margin = 1in]{geometry}
\usepackage{amsmath, amssymb}
\usepackage{float}

\usepackage{amsthm}
\theoremstyle{plain}
\newtheorem{theorem}{Theorem}[section]
\newtheorem{lemma}[theorem]{Lemma}
\newtheorem{proposition}[theorem]{Proposition}
\theoremstyle{definition}
\newtheorem{definition}[theorem]{Definition}
\newtheorem{example}[theorem]{Example}
\theoremstyle{remark}
\newtheorem{remark}[theorem]{Remark}
\newtheorem{algorithm}[theorem]{Algorithm}

\usepackage[style = alphabetic, maxalphanames=4, maxnames = 4, doi = false, url = false, isbn = false, firstinits = true]{biblatex}
\DeclareFieldFormat[article, misc, incollection, inbook, inproceedings, book]{title}{#1}

\usepackage{quiver}
\usepackage{tikz}

\newcommand{\spec}{\mathrm{Spec}\,}
\newcommand{\proj}{\mathrm{Proj}\,}
\newcommand{\gr}{\mathrm{gr}}
\newcommand{\coker}{\mathrm{coker}}
\newcommand{\Div}{\mathrm{Div}}

\newcommand{\ord}{\mathrm{ord}}
\newcommand{\comp}{\mathrm{Comp}}
\newcommand{\sing}{\mathrm{Sing}}
\newcommand{\inc}{\mathrm{Inc}}
\newcommand{\br}{\mathrm{Br}}

\newcommand{\onto}{\twoheadrightarrow}
\newcommand{\into}{\hookrightarrow}

\title{Defining reduction types of curves via minimal regular and minimal normal crossings models}
\author{Jakab Schrettner}
\begin{document}
\maketitle
\begin{abstract}
	We propose a definition of the reduction type of a curve over a discretely valued field in terms of the special fibre of an arbitrary regular model. We show that under this definition, the reduction type in terms of the minimal regular model determines and reduction type in terms of the minimal regular normal crossings model and vice versa. We also show that our definition is compatible with earlier classification results in low genus, namely the Kodaira-Néron classification for elliptic curves and the classification result of Namikawa-Ueno in genus 2. The essential new element in our definition is a suitable invariant of the singularities on the special fibre, building on the notion of equisingularity introduced by Zariski.
\end{abstract}
\tableofcontents
\section{Introduction}
Let $K$ be a discretely valued field with valuation ring $\mathcal{O}_K$ and algebraically closed residue field $k$. In the study of (smooth, projective) curves $C$ over $K$, one often considers their reduction type: this is a discrete invariant given by a combinatorial description of the special fibre of a `nice' choice of model $X\to \spec \mathcal{O}_K$ of $C$. In many cases, $X$ is chosen to be the minimal regular model of $C$: this is done for example in the classical Kodaira-Néron classification of reduction types of elliptic curves \cite{Kodaira} \cite{Neron}, and the classification of reduction types of genus 2 curves by Namikawa-Ueno \cite{namikawa-ueno} (building on an earlier `numerical' classification by Ogg \cite{ogg}). One can however also choose to classify curves by the special fibre of a minimal regular normal crossings model, see for example the work of Dokchitser \cite{tim_classification} and the genus 2 classification of Aylward, Kellock, Dokchitser and Lupoian \cite{genus_2_tables}.

Specifying the reduction type of a curve $C$ involves making a choice of model $X$ of $C$ and specifying in some way the following data:
\begin{itemize}
	\item The combinatorial configuration of components in the special fibre: by this we mean a list of components $\Gamma_1, \ldots, \Gamma_r$ of the special fibre, and specifying which ones intersect and in how many points,
	\item Some numerical data about the components: by this we mean the multiplicity of each component in the special fibre, the genus of the normalisation of each component (and/or its arithmetic genus), and the intersection numbers $\Gamma_i\cdot \Gamma_j$,
	\item Some information about the singularities on the components: whether a singular point is a cusp or a node, etc.
\end{itemize}
If $X$ is the minimal normal crossings model, this can be made precise to define a notion of reduction type, see \cite[Definition 3.6]{tim_classification}. However if $X$ is the minimal regular model, such a definition is (to the author's knowledge) not available in arbitrary genus, essentially because there is no classification of possible singularities that can arise on the special fibre. In low genus this can be handled in an ad hoc way: in the genus 1 Kodaira-Néron classification, the only relevant distinction is the one between a cusp and a node on a rational curve. In genus 2 there are some more possibilities: in \cite{namikawa-ueno} one can find, in addition to cusps and nodes, a reduction type involving a $(2,5)$-cusp (type $\mathrm{VIII-1}$ in loc.cit), and a singularity where two branches are tangent (type $\mathrm{III-II}_0$, which we will call a \emph{tacnode}).

Our main result is a definition of reduction type in terms of an arbitrary regular model $X$, which keeps track of the combinatorial configuration of components on the special fibre, numerical data and appropriate invariants of each singularity. Specifically for each singularity we record the set of its branches, the local intersection numbers between branches and an invariant called the \emph{valuation semigroup} associated to each branch.
\begin{definition}[= Definition \ref{definition_type}]
	An \emph{(abstract) type} is a tuple $T = (\comp, d, g, \sing, \inc, \br, S, i)$ where
	\begin{enumerate}
		\item $\comp$ is a nonempty finite set whose elements $\Gamma$ are called components, and $d:\comp \to \mathbb{Z}_{>0}$, $g:\comp \to \mathbb{Z}_{\ge 0}$ are functions, which we think of as associating to $\Gamma$ its multiplicity $d_\Gamma$ and (geometric) genus $g_\Gamma$.
		\item $\sing$ is a finite set whose elements $p$ are called singular points, along with an incidence relation $\inc\subseteq \sing\times\comp$; if $(p, \Gamma)\in \inc$ we say $p$ is on $\Gamma$ and write $p\in \Gamma$.
		\item $\br = (\br_{\Gamma, p})_{\Gamma,p}$ is a collection of nonempty finite sets $\br_{\Gamma, p}$ for all components $\Gamma$ and singular points $p$ with $p\in \Gamma$. Elements $B$ of $\br_{\Gamma, p}$ are  branches of $\Gamma$ through $p$. The function $S$ associates to each branch $B\in  \coprod\br_{\Gamma, p}$ a numerical semigroup $S(B)$ which we think of as the valuation semigroup of the branch $B$. The function $i$ associates to each pair of distinct branches $B,B'$ through the same point $p$ an intersection number $i_p(B, B')$, which is a positive integer. 
	\end{enumerate}
	(For the definition of a numerical semigroup see Definition \ref{definition_numerical_semigroup})
	
	The \emph{type} of an arithmetic surface $X$ over $\mathcal{O}_K$ is the abstract type associated to the special fibre $X_s$ in a natural way (Definition \ref{def_type_of_surface}). If $C$ is a smooth projective curve over $K$, we define its \emph{minimal regular type} and \emph{minimal normal crossings type} to be the type of its minimal regular and minimal regular normal crossings model, respectively.
\end{definition}
\begin{example}\label{example:low_genus_singularities}
	If $X$ is the minimal regular model of an elliptic curve of Kodaira type $\mathrm{I}_1$, then the special fibre $X_s$ consists of a single rational (genus 0) curve of multiplicity 1, with a node $p$. Then the type of $X$ has 1 component $\Gamma$ with a unique singular point $p$ lying on it. Moreover $\Gamma$ has two branches $B_1, B_2$ through $p$ which intersect with multiplicity $1 = i_p(B_1, B_2)$, and their valuation semigroups are $S(B_1) = S(B_2) = \mathbb{Z}_{>0}$.
	
	If $X$ is the minimal regular model of an elliptic curve of Kodaira type $\mathrm{II}$, then the special fibre $X_s$ consists of a single rational (genus 0) curve of multiplicity 1, with a cusp $p$. Then the type of $X$ has 1 component $\Gamma$ with a unique singular point $p$ lying on it. However $\Gamma$ only has one branch $B$ through $p$, whose valuation semigroup is $S(B) = \left\langle 2,3\right\rangle = \{2a+3b\;:\;a,b\in \mathbb{Z}_{>0}\}$.
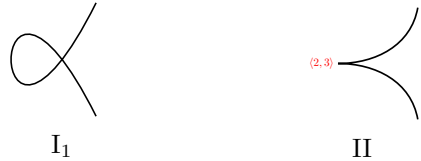
		\begin{figure}[H]
			\centering
			\tikz{\node[label=-90:$\mathrm{I}_1$](1){\resizebox{!}{1.5cm}{\begin{tikzpicture}
											\clip (-1.6,-1) rectangle + (2, 2);
											\draw[thick] (0, 1) .. controls (-2,-3) and (-2, 3) .. (0, -1);
								\end{tikzpicture}}};}
			\hspace*{2cm}
			\tikz{\node[label=-90:$\mathrm{II}$](2){\resizebox{!}{1.5cm}{\begin{tikzpicture}
											\draw[semithick] (1,0.7) to [out = -100, in = 0] (0,0) to [out = 0, in = 100] (1, -0.7);
											\node[left, red, scale = 0.4] at (0,0) {$\left\langle 2,3\right\rangle$};
								\end{tikzpicture}}};}
		\end{figure}
	The $(2,5)$-cusp in Namikawa-Ueno type $\mathrm{VIII-I}$ is a germ with a single branch whose valuation semigroup is $\left\langle 2,5\right\rangle = \{2a+5b\;:\;a,b\in \mathbb{Z}_{>0}\}$. 
	
	The singularity in Namikawa-Ueno type $\mathrm{III-II}_0$ is one with two branches, both smooth (hence having valuation semigroup $S = \mathbb{Z}_{>0}$), and they intersect with multiplicity 2 (we call this a \emph{tacnode}).
	
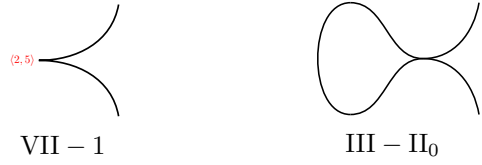
\begin{figure}[H]
		\centering
		\tikz{\node[label=-90:$\mathrm{VII-1}$](2){\resizebox{!}{1.5cm}{\begin{tikzpicture}
						\draw[semithick] (1,0.7) to [out = -100, in = 0] (0,0) to [out = 0, in = 100] (1, -0.7);
						\node[left, scale = 0.4, red] at (0,0) {$\left\langle 2,5\right\rangle$};
			\end{tikzpicture}}};}
		\hspace*{2cm}
		\tikz{\node[label=-90:$\mathrm{III}-\mathrm{II}_0$](2){\resizebox{!}{1.5cm}{\begin{tikzpicture}
						\draw[thick] (1,1) to [out = -100, in = 0] (0,0);
						\draw[thick] (1, -1) to [out = 100, in = 0] (0,0);
						\draw[thick] (0,0) to [out = -180, in = 0] (-1.3, 1);
						\draw[thick] (0,0) to [out = -180, in = 0] (-1.3, -1);
						\draw[thick] (-1.3, 1) to [out = -180, in = -180] (-1.3, -1);
			\end{tikzpicture}}};}
	\end{figure}
\end{example}
To justify this definition, we show that the two notions of type (minimal regular and minimal normal crossings) are determined by one another, and that in low genus (genus 1 and 2) one recovers the existing classification results of Kodaira-Néron and Namikawa-Ueno in terms of the minimal regular model. Our main theorems are hence the following:
\begin{enumerate}
	\item (Theorem \ref{min_reg<->rnc}) Given a smooth projective curve $C$ over $K$, its minimal regular reduction type determines its minimal normal crossings reduction type and vice versa. In fact both of them are determined by the type of any regular model $X$ of $C$. Moreover we can give a procedure to determine one in terms of the other (Algorithm \ref{algorithm_blowup}). Hence we can call either of these the \emph{reduction type} of $C$.
	\item (Theorem \ref{genus_1_classification}, Theorem \ref{genus_2_classification}) One can classify curves of genus 1 and 2 by their minimal regular reduction type. The minimal regular reduction type of an elliptic curve is a Kodaira type, and that of a general genus 1 curve is a multiple of a Kodaira type. The minimal regular reduction types of genus 2 curves are given as in Namikawa-Ueno \cite{namikawa-ueno}.
\end{enumerate}
Hence classifying curves in terms of their minimal regular models is equivalent to classifying them by normal crossings models, and one can translate between these two notions easily. Note also that Namikawa-Ueno in \cite{namikawa-ueno} work over $\mathbb{C}$ and obtain their classification using analytic methods; our result shows that the resulting classification of reduction types still holds over an arbitrary discrete valuation ring (and in particular in arbitrary residue characteristic).

To show (1) above, we note that the minimal regular normal crossings model is obtained from the minimal regular model by a finite sequence of blowups at closed points. We therefore study how the type changes upon blowing up at a closed point. We prove (Theorem \ref{blowup_type}) that under a blowup $\widetilde{X}\to X$ at a closed point, the type of $\widetilde{X}$ determines the type of $X$ and vice versa. In fact this shows something more general: that the minimal regular/minimal normal crossings reduction types are determined by the type of any regular model (since any two regular models are related by sequences of blowups and blow-downs).

To show (2), we use Dokchitser's classification of normal crossings reduction types \cite{tim_classification}, then use our results to deduce the corresponding minimal regular types in each case.

\begin{example}\label{example_genus_3}
	Consider the curves $C_1: y^2 = x^7 + p$ and $C_2: y^3 = x^4 + p$ over $K = \mathbb{Q}_p^{\mathrm{ur}}$. Both of these are smooth of genus 3 (one is hyperelliptic, the other is a plane quartic), and the equations given define regular models $\mathcal{C}_1$ and $\mathcal{C}_2$. For both of them the special fibre consists of a single component with a unique unibranch singularity at the origin (and hence both models are minimal regular), and their semigroups are $S_1 = \left\langle 2,7\right\rangle$ and $S_2 = \left\langle 3,4\right\rangle$ respectively. Performing blowups to obtain normal crossings models, we obtain the following reduction types:
	
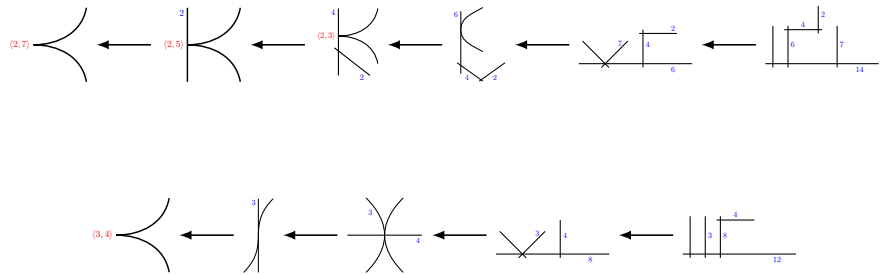
\begin{figure}[H]
		\centering
		\tikz[remember picture]{\node(1){%
				\resizebox{!}{1cm}{%
					\begin{tikzpicture}
						\draw[thick] (1,0.7) to [out = -100, in = 0] (0,0) to [out = 0, in = 100] (1, -0.7);
						\node[left, scale = 0.5, red] at (0,0){$\left\langle2,7\right\rangle$};
			\end{tikzpicture}}};}
		\hspace*{0.5cm}
		\tikz[remember picture]{\node(2){%
				\resizebox{!}{1cm}{%
					\begin{tikzpicture}
						\draw[thick] (0, -0.7) to (0, 0.7);
						\draw[thick] (1, 0.7) to [out = -100, in = 0] (0,0) to [out = 0, in = 100] (1, -0.7);
						\node[left, scale = 0.5, red] at (0,0){$\left\langle 2,5\right\rangle$};
						\node[left, scale = 0.5, blue] at (0, 0.6){2};
			\end{tikzpicture}}};}
		\hspace*{0.5cm}
		\tikz[remember picture]{\node(3){%
				\resizebox{!}{1cm}{%
					\begin{tikzpicture}
						\draw[thick] (0,0.7) to (0, -1);
						\draw[thick] (1, 0.7) to [out = -100, in = 0] (0,0) to [out = 0, in = 100] (1, -0.7);
						\draw[thick] (-0.1, -0.3) to (0.8, -1);
						\node[left, scale = 0.6, red] at (0,0) {$\left\langle 2,3\right\rangle$};
						\node[left, scale = 0.6, blue] at (0, 0.6){4};
						\node[below, scale = 0.6, blue] at (0.6, -0.9){2};				
			\end{tikzpicture}}};}
		\hspace*{0.5cm}
		\tikz[remember picture]{\node(4){%
				\resizebox{!}{1cm}{%
					\begin{tikzpicture}
						\draw[thick] (0, 0.7) to (0, -1);
						\draw[thick] (0.6, 0.8) to [out = -150, in = 90] (0,0.2) to [out = -90, in = 150] (0.6, -0.4);
						\draw[thick] (-0.1, -0.7) to (0.6, -1.2);
						\draw[thick] (0.5, -1.2) to (1.2, -0.7);
						\node[left, scale = 0.6, blue] at (0, 0.6){6};
						\node[left, scale = 0.6, blue] at (0.3, -1.1){4};
						\node[right, scale = 0.6, blue] at (0.8, -1.1){2};
			\end{tikzpicture}}};}
		\hspace*{0.5cm}
		\tikz[remember picture]{\node(5){%
				\resizebox{!}{1cm}{%
					\begin{tikzpicture}
						\clip (0,-0.5) rectangle + (3, 2);
						\draw[thick] (0,0) to (3,0);
						\draw[thick] (0.8, -0.1) to (0.1, 0.6);
						\draw[thick] (0.6, -0.1) to (1.3, 0.6);
						\draw[thick] (1.7, -0.1) to (1.7, 0.9);
						\draw[thick] (1.6, 0.8) to (2.6, 0.8);
						\node[below, blue, scale = 0.6] at (2.5, 0) {6};
						\node[above, blue, scale = 0.6] at (1.1, 0.4){7};
						\node[right, blue, scale = 0.6] at (1.7, 0.5){4};
						\node[above, blue, scale = 0.6] at (2.5, 0.8){2};
			\end{tikzpicture}}};}
		\hspace*{0.5cm}
		\tikz[remember picture]{\node(6){%
				\resizebox{!}{1cm}{%
					\begin{tikzpicture}
						\clip (0,-0.5) rectangle + (3, 2);
						\draw[thick] (0,0) to (3,0);
						\draw[thick] (0.2, -0.1) to (0.2, 1);
						\draw[thick] (0.6, -0.1) to (0.6, 1);
						\draw[thick] (0.5, 0.9) to (1.5, 0.9);
						\draw[thick] (1.4, 0.8) to (1.4, 1.8);
						\draw[thick] (1.9, -0.1) to (1.9, 1);
						\node[right, blue, scale = 0.6] at (0.6, 0.5) {6};
						\node[above, blue, scale = 0.6] at (1, 0.9) {4};
						\node[right, blue, scale = 0.6] at (1.4, 1.3) {2};
						\node[right, blue, scale = 0.6] at (1.9, 0.5) {7};
						\node[below, blue, scale = 0.6] at (2.5, 0) {14};
			\end{tikzpicture}}};}
		\tikz[overlay, remember picture]{
			\draw[latex-, semithick] (1) -- (1-|2.west);
			\draw[latex-, semithick] (2) -- (2-|3.west);
			\draw[latex-, semithick] (3) -- (3-|4.west);
			\draw[latex-, semithick] (4) -- (4-|5.west);
			\draw[latex-, semithick] (5) -- (5-|6.west);
		}
	\end{figure}
	\vspace*{0.5cm}
	\begin{figure}[H]
		\centering
		\tikz[remember picture]{\node(1){%
				\resizebox{!}{1cm}{%
					\begin{tikzpicture}
						\draw[thick] (1,0.7) to [out = -100, in = 0] (0,0) to [out = 0, in = 100] (1, -0.7);
						\node[left, scale = 0.5, red] at (0,0){$\left\langle3,4\right\rangle$};
			\end{tikzpicture}}};}
		\hspace*{0.5cm}
		\tikz[remember picture]{\node(2){%
				\resizebox{!}{1cm}{%
					\begin{tikzpicture}
						\draw[thick] (0, -1) to (0, 1);
						\draw[thick] (-0.4, -1) to [out = 45, in = -90] (0,0) to [out = 90, in = -135] (0.4, 1);
						\node[left, scale = 0.6, blue] at (0, 0.9) {3};
			\end{tikzpicture}}};}
		\hspace*{0.5cm}
		\tikz[remember picture]{\node(3){%
				\resizebox{!}{1cm}{%
					\begin{tikzpicture}
						\draw[rotate = -90, thick] (-1, 0.5) parabola bend(0,0) (1, 0.5);
						\draw[rotate = 90, thick] (-1, 0.5) parabola bend(0,0) (1, 0.5);
						\draw[thick] (-1, 0) to (1,0);
						\node[left, scale = 0.6, blue] at (-0.25, 0.6) {3};
						\node[below, scale = 0.6, blue] at (0.9, 0) {4};				
			\end{tikzpicture}}};}
		\hspace*{0.5cm}
		\tikz[remember picture]{\node(4){%
				\resizebox{!}{1cm}{%
					\begin{tikzpicture}
						\clip (0,-0.5) rectangle + (3, 2);
						\draw[thick] (0,0) to (3,0);
						\draw[thick] (0.8, -0.1) to (0.1, 0.6);
						\draw[thick] (0.6, -0.1) to (1.3, 0.6);
						\draw[thick] (1.7, -0.1) to (1.7, 0.9);
						\node[below, blue, scale = 0.6] at (2.5, 0) {8};
						\node[above, blue, scale = 0.6] at (1.1, 0.4){3};
						\node[right, blue, scale = 0.6] at (1.7, 0.5){4};
			\end{tikzpicture}}};}
		\hspace*{0.5cm}
		\tikz[remember picture]{\node(5){%
				\resizebox{!}{1cm}{%
					\begin{tikzpicture}
						\clip (0,-0.5) rectangle + (3, 2);
						\draw[thick] (0,0) to (3,0);
						\draw[thick] (0.2, -0.1) to (0.2, 1);
						\draw[thick] (0.6, -0.1) to (0.6, 1);
						\draw[thick] (1, -0.1) to (1, 1);
						\draw[thick] (0.9, 0.9) to (1.9, 0.9);
						\node[right, blue, scale = 0.6] at (0.6, 0.5) {3};
						\node[right, blue, scale = 0.6] at (1, 0.5) {8};
						\node[above, blue, scale = 0.6] at (1.4, 0.9) {4};
						\node[below, blue, scale = 0.6] at (2.5, 0) {12};
			\end{tikzpicture}}};}
		\tikz[overlay, remember picture]{
			\draw[latex-, semithick] (1) -- (1-|2.west);
			\draw[latex-, semithick] (2) -- (2-|3.west);
			\draw[latex-, semithick] (3) -- (3-|4.west);
			\draw[latex-, semithick] (4) -- (4-|5.west);
		}
	\end{figure}
	Hence the normal crossings types of these curves are given as in the pictures on the right.
\end{example}

The essential ingredient in our definition of types is a suitable invariant of a singularity found on the special fibre. It is natural to study a singularity in terms of the completed local ring $\widehat{\mathcal{O}}_{X_s, p}$ of the special fibre at a point $p$. It turns out that this local ring is always planar, i.e. there is a surjection $k[[x,y]]\onto \widehat{\mathcal{O}}_{X_s, p}$, hence the relevant singularities are plane curve singularities. We are thus led to the study of \emph{plane curve germs}, i.e. rings of the form $k[[x,y]]/(f)$ for a non-unit power series $f$.

Let us briefly describe the invariant we use: it is the the \emph{equisingularity type} of the germ of a plane curve, introduced by Zariski in \cite{zariski1986probleme} (see also the English translation \cite{zariski2006moduli}). For irreducible germs (branches), the equisingularity types correspond to components of the moduli space of plane curve germs up to analytic isomorphism. They are encoded by an invariant called the \emph{valuation semigroup} (Definition \ref{definition_semigroup}), which is a subset $S\subset \mathbb{Z}_{>0}$ of the positive integers closed under addition. For general germs, we decompose them into branches (irreducible components), and keep track of the valuation semigroup of each branch, as well as the intersection multiplicity of any pair of distinct branches. This is an invariant that `sees' the combinatorics of the resolution process of the singularity, in order to satisfy property (1) above, and is also an invariant that is discrete so that it can be used for classification purposes as in property (2).

\subsection*{Layout}
In Section \ref{section_plane_curves} we study plane curve germs: we define plane curve germs and branches, the branches of a plane curve germ, the valuation semigroup and the intersection numbers of branches. We will also look at how all these invariants change under blowing up. The definition of the semigroup goes back to Zariski \cite{zariski1986probleme} in the case $k = \mathbb{C}$, however we need a definition that works in arbitrary characteristic, so we present things slightly differently, mostly citing the results of Płoski \cite{plane_algebraic_curves}, \cite{poski2014plane}.

In Section \ref{section_global_curves}, we relate this to the study of projective curves $Y$ over $k$ (`global curves'). We show that when $Y$ is locally planar (meaning that its complete local rings are quotients of $k[[x,y]]$), then invariants of $Y$ at points (multiplicity, genus, intersection numbers) depend only on the invariants of the germ of $Y$ at $p$ defined in the previous section. We will also use this to classify plane curve singularities that can arise on curves of low genus: we show that the only possible singularities on integral curves of genus $\le 2$ are those seen in Example \ref{example:low_genus_singularities}, namely nodes, cusps, $(2,5)$-cusps and tacnodes.

In Section \ref{section_arithmetic_surfaces} we define the type of an arithmetic surface and the minimal regular/minimal normal crossings reduction type, and we prove that one determines the other. We do this by studying how blowups at closed points change the type of an arithmetic surface.

Finally in Section \ref{section_low_genus} we show that our definition of reduction type recovers the existing classifications of Kodaira-Néron in genus 1 and Namikawa-Ueno in genus 2.
\subsection*{Notation and conventions}
Throughout $k$ will denote an algebraically closed field. In Sections \ref{section_arithmetic_surfaces} and \ref{section_low_genus} we will work over a discrete valuation ring $\mathcal{O}_K$ with fraction field $K$ and residue field $k$ (which will still be assumed algebraically closed).

The scheme $\spec \mathcal{O}_K$ has a closed point $s$, and a generic point $\eta$. If $X$ is an $\mathcal{O}_K$-scheme, the fibre above $s$ is the special fibre $X_s$, and the fibre above $\eta$ is the generic fibre $X_\eta$. 

An \emph{arithmetic surface} is a regular proper flat $\mathcal{O}_K$-scheme $X\to \spec \mathcal{O}_K$ of relative dimension 1. By a \emph{model} of a smooth projective curve $C$ over $K$ we mean a proper flat $\mathcal{O}_K$-scheme $X$ along with an isomorphism $X_\eta \cong C$. In this paper we will only consider regular models, where $X$ is assumed to be a regular scheme, and is therefore an arithmetic surface. By a regular \emph{normal crossings} model $X$ we mean one where the special fibre $X_s$ is a normal crossings divisor in the sense of \cite[Definition 9.1.16]{liu}, in particular the special fibre is not assumed to be reduced.

For a power series $f\in k[[t]]$ or $f\in k[[x,y]]$, its \emph{order} is the smallest degree of a monomial appearing in $f$ with nonzero coefficient. Sometimes we'll write this as $\ord\,f$ or (if $f\in k[[t]]$) as $\ord_t\,f$. When we write $f$ as a sum of homogeneous terms, we mean a decomposition of the form
\[f = f_\mu + f_{\mu+1}+\ldots\]
where each $f_i$ is homogeneous of degree $i$, and $\mu$ is the order of $f$. The \emph{initial form} of $f$ is $f_\mu$, the smallest nonzero homogeneous term in $f$. A \emph{linear form} $l(x,y)$ is an expression of the form $l(x,y) = ax + by$ for $a,b\in k$ not both 0.

For any scheme $X$ and a point $p\in X$, we denote by $\mathcal{O}_{X,p}$ the stalk of the structure sheaf of $X$ at $p$, and by $\widehat{\mathcal{O}}_{X,p}$ its completion with respect to its maximal ideal. We denote by $X_{red}$ the underlying reduced closed subscheme.

We denote by $\mathbb{Z}_{\ge 0}$ the set of nonnegative integers and by $\mathbb{Z}_{>0}$ the set of positive integers.
\subsection*{Acknowledgements}
I would like to thank Vladimir Dokchitser for suggesting this problem to me and for helpful discussions and guidance throughout the development of this paper. I would also like to thank the participants of the \emph{Tamagawa Numbers of curves} research meeting for their feedback on an early version of this work.
\section{Plane curve germs}\label{section_plane_curves}
Let $k$ be an algebraically closed field. In this section we define and study plane curve germs, which are germs of plane curves near a point. First we study branches, which are irreducible germs, and define the valuation semigroup (Definition \ref{definition_semigroup}). Then we look at reducible germs, and define the intersection number of two distinct branches (Definition \ref{definition_intersection}) and also the notion of tangency of branches (Definition \ref{definition_tangency}). Finally, we look at how all these change upon blowing up, the result being described by Theorems \ref{blowup_germ} and \ref{blowup_intersection}.
\begin{definition}
	A \emph{plane curve germ} (over $k$) is an affine scheme $C = \spec A$ where $A$ is a 1-dimensional complete local $k$-algebra of the form $A = k[[x,y]]/(f)$ for some $f\in k[[x,y]]$. The \emph{branches} of a plane curve germ $C$ are the schemes $\spec A/\mathfrak{p}$ where $\mathfrak{p}$ is a minimal prime in $A$ (i.e. the irreducible components of $C$). A \emph{plane curve branch} is a plane curve germ that is integral (or equivalently, reduced with just one branch).
\end{definition}
Given a plane curve germ $C = \spec k[[x,y]]/(f)$ and a factorisation $f = f_1^{a_1}f_2^{a_2}\ldots f_r^{a_r}$ of $f$ into irreducibles, the branches of $C$ are exactly $B_i = \spec k[[x,y]]/(f_i)$, and $C$ is a branch if and only if $f$ is irreducible. It's also clear to see that the underlying reduced scheme is $C_{red} = \spec k[[x,y]]/(f_1f_2\ldots f_r)$ is also a plane curve germ, and $C$ is reduced if and only if $a_i=1$ for all $i$. Moreover $C_{red}$ has the same branches as $C$ does.
\subsection{Plane curve branches}
Plane curve branches are of the form $\spec k[[x,y]]/(f)$ where $f$ is an irreducible power series.

The following lemma says that a plane curve branch can only have one `tangent direction'.
\begin{lemma}\label{tangent_direction}
	Suppose $f\in k[[x,y]]$ is an irreducible power series of order $\mu$. Write $f$ as a sum of homogeneous terms as
	\[f = f_\mu + f_{\mu+1}+\ldots\;.\]
	Then there is a linear form $l(x,y) = ax+by$, $a,b\in k$ such that $f_\mu = l^\mu$.
\end{lemma}
\begin{proof}
	Since $k$ is algebraically closed, $f_\mu$ factors as a product of linear forms. If there are two distinct factors among these, then \cite[Factorization lemma]{plane_algebraic_curves} implies that $f$ can also be factored as a product of two non-units, contradicting irreducibility.
\end{proof}
Thus given an irreducible power series $f$, we can change coordinates so that its initial form becomes e.g. $y^\mu$. We record this as a definition:
\begin{definition}
	Let $f\in k[[x,y]]$ be an irreducible power series of order $\mu$. We say $f$ is \emph{$y$-aligned} if its initial form is $y^\mu$. In this case $f$ can be written as
	\[f = y^\mu + a_1(x)y^{\mu-1}+\ldots + a_{\mu-1}(x)y + a_{\mu}(x)\]
	where the $a_i(x)\in k[[x]]$ are power series, and $a_i(x)$ has order $\ge i + 1$.
\end{definition}
Intuitively, a $y$-aligned power series is one that defines an irreducible curve germ whose only tangent line is the $x$-axis.

It follows from the above discussion that any plane curve branch can be written as $\spec k[[x,y]]/(f)$ for a $y$-aligned $f$ (pick $f$ arbitrary, with initial form $l^\mu$, then apply a linear automorphism of $k[[x,y]]$ which sends $l\mapsto y$).

We will also need the following result:
\begin{theorem}\label{branch_normalisation}
	Let $C = \spec A$ be a plane curve branch. Then the normalisation of $A$ (i.e. its integral closure in its field of fractions) is isomorphic to a power series ring $k[[t]]$ in one variable. Moreover the quotient $k[[t]]/A$ is finite-dimensional over $k$.
\end{theorem}
\begin{proof}
	The normalisation $A'$ of $A$ is a local ring by \cite[Lemma 15.107.3, Tag 0BQ0]{stacks-project}. It is of dimension 1 \cite[Lemma 10.112.4, Tag 00OK]{stacks-project}. So $A'$ is a regular complete local ring of dimension 1, containing the field $k$. By the Cohen structure theorem, $A'\cong k[[t]]$.
	
	Since $A$ is a Nagata ring \cite[Lemma 10.162.8, Tag 032W]{stacks-project}, the normalisation $k[[t]]$ is finite as an $A$-module \cite[Lemma 29.54.11, Tag 035S]{stacks-project}. Hence $k[[t]]/A$ is a finitely generated torsion $A$-module, and must have finite $k$-dimension as claimed.
\end{proof}
Hence we can define the following:
\begin{definition}
	Let $C$ be a plane curve branch. We say that $C$ is \emph{smooth} if $C$ is a regular scheme, and \emph{singular} otherwise. 
\end{definition}
Since $C$ is one-dimensional, regularity is equivalent to normality, so $C$ is smooth if and only if $C\cong \spec k[[t]]$ by the above theorem.

We also have the following refinement of Lemma \ref{tangent_direction}:
\begin{proposition}\label{cotangent_line}
	Let $C = \spec A$ be a plane curve branch and $\mathfrak{m}\subseteq A$ the maximal ideal. Consider the associated graded ring
	\[\gr_{\mathfrak{m}}(A) = \bigoplus_{n\ge 0} \mathfrak{m}^n/\mathfrak{m}^{n+1}.\]
	Then
	\begin{enumerate}
		\item If $C$ is smooth, then the cotangent space $\mathfrak{m}/\mathfrak{m}^2$ is 1-dimensional over $k$, and if $t$ is any generator of it then
		\[\gr_{\mathfrak{m}}(A)\cong k[t]\]
		is a polynomial ring.
		\item If $C$ is singular, then the cotangent space $\mathfrak{m}/\mathfrak{m}^2$ is 2-dimensional over $k$, and there is a unique line $L\subseteq \mathfrak{m}/\mathfrak{m}^2$ whose elements are nilpotent in $\gr_{\mathfrak{m}}(A)$.
	\end{enumerate}
\end{proposition}
\begin{proof}
	We have a surjection $k[[x,y]]\onto A$ which means $\dim_k \mathfrak{m}/\mathfrak{m}^2\le 2$. Since $A$ is 1-dimensional, the dimension is 1 if $C$ is smooth and 2 otherwise. The statement in the smooth case is \cite[Theorem 11.22(i)]{atiyah2018introduction}.
	
	In the singular case, we have a surjection
	\[k[[x,y]]\onto A\]
	so that $A\cong k[[x,y]]/(f)$ for some power series $f$. This induces a surjection on associated graded rings
	\[k[x,y]\onto \gr_{\mathfrak{m}}A\]
	(here we use that the associated graded ring of $k[[x,y]]$ is $k[x,y]$, by the reference above). The kernel consists of all power series which occur as the initial form of some element of the ideal $(f)$. It's easy to see that these are exactly the multiples of the initial form $f_{in}$, so 
	\[\gr_{\mathfrak{m}}A\cong k[x,y]/(f_{in}).\]
	By Lemma \ref{tangent_direction} we know that $f_{in} = l^\mu$ for some linear form $l$ and integer $\mu\ge 1$. So the image of $l$ in $\mathfrak{m}/\mathfrak{m}^2$ is nilpotent. It is also nonzero since $C$ is singular which means that
	\[(x,y)/(x,y)^2\cong \mathfrak{m}/\mathfrak{m}^2\]
	as $k$-vector spaces, and $l$ is non-zero in the former. This shows that the line $L$ generated by $l$ in the cotangent space consists of nilpotent elements. Uniqueness of this line is clear as otherwise $\gr_{\mathfrak{m}}A\cong k[x,y]/(f_{in})$ would be 0-dimensional, which is not possible.
\end{proof}
We also consider parametrisations of plane curve branches:
\begin{definition}
	Let $C = \spec A$ be a plane curve branch. A \emph{parametrisation} of $C$ is an injective ring map 
	\[A\into k[[t]]\]
	such that $\dim_k k[[t]]/A$ is finite.
\end{definition}
Theorem \ref{branch_normalisation} ensures that parametrisations exist for any plane curve branch. Any two parametrisations realise the normalisation of $A$, and therefore differ by post-composing with an automorphism of $k[[t]]$.

If we write $A = k[[x,y]]/(f)$, then such a map is determined by the images of $x,y$ respectively. If $x\mapsto \varphi(t)$ and $y\mapsto \psi(t)$ then $\varphi(t), \psi(t)$ are power series such that
\[f(\varphi(t), \psi(t)) = 0.\]
So we can think of $t\mapsto (\varphi(t), \psi(t))$ as ``parametrising'' the branch $f = 0$.
\begin{example}\label{example_parametrisation}
	Let $f = x^m - y^n$ where $m,n$ are coprime positive integers and $A = k[[x,y]]/(f)$. Then $A\into k[[t]]$ via
	\[x\mapsto t^n,\qquad y\mapsto t^m\]
	making $A = k[[t^m, t^n]]$. Since $m,n$ are coprime, the image contains $t^N$ for $N$ sufficiently large, so $k[[t]]/A$ is finite-dimensional. This shows that $A$ is integral (i.e. $f$ irreducible and $C = \spec A$ is a branch), and the maps above give a parametrisation of $C$.
\end{example}
We have the following general fact about parametrisations:
\begin{proposition}\label{parametrisation}
	Let $f\in k[[x,y]]$ be a $y$-aligned power series of order $\mu$. Then for any parametrisation
	\[k[[x,y]]/(f)\into k[[t]]\]
	sending $x\mapsto \varphi(t), y\mapsto \psi(t)$, $\varphi$ has order $\mu$ and $\psi$ has order $>\mu$.
\end{proposition}
\begin{proof}
	We write
	\[f = y^\mu + a_1(x)y^{\mu-1}+\ldots + a_{\mu-1}(x)y + a_\mu(x)\]
	with $a_i(x)$ of order $\ge i + 1$.
	
	Pick any parametrisation $k[[x,y]]/(f)\into k[[t]]$, sending $x\mapsto \varphi(t), y\mapsto \psi(t)$. Then we have
	\[\psi(t)^\mu + a_1(\varphi(t))\psi(t)^{\mu-1} + \ldots + a_{\mu-1}(\varphi(t))\psi(t) + a_\mu(\varphi(t)) = 0.\]
	Which we can write as
	\[\psi(t)^\mu + \sum_{i+j\ge \mu + 1}c_{ij}\psi(t)^i\varphi(t)^j = 0\]
	for some $c_{ij}\in k$. If $\ord \,\varphi \ge \ord\, \psi$, then the leading term on the left-hand side is the leading term of $\psi^\mu$, which is nonzero, contradiction the equality. So $\ord \,\varphi < \ord\, \psi$.
	
	By \cite[Theorem 3.14]{plane_algebraic_curves}, we have that 
	\[\ord\, \varphi(t) = \dim_k k[[x,y]]/(x,f) = \dim_k k[[y]]/(f(0,y)) = \mu\]
	as claimed.
\end{proof}
The main invariant we're interested in for plane branches is the valuation semigroup, which determines the equisingularity type in this case.
\begin{definition}\label{definition_semigroup}
	Let $C = \spec A$ be a plane curve branch. We have a module-finite extension $A\subseteq k[[t]]$ given by normalisation. Let $v = \ord:k[[t]]\setminus\{0\} \to \mathbb{Z}_{\ge0}$ be the canonical valuation. The \emph{valuation semigroup} of $C$ (or $A$) is the set
	\[S(C) = v(A\setminus A^\times)\subseteq \mathbb{Z}_{>0}.\]
\end{definition}
The justification for the name is that $S(C)$ is a numerical semigroup:
\begin{definition}\label{definition_numerical_semigroup}
	A \emph{numerical semigroup} is a set $S\subseteq \mathbb{Z}_{>0}$ which is closed under addition and such that $\mathbb{Z}_{>0}\setminus S$ is finite.
\end{definition}
If $(a_1, \ldots, a_n)$ are positive integers with $\gcd(a_1, \ldots, a_n) = 1$, then we denote by $\left\langle a_1, \ldots, a_n\right\rangle$ the semigroup generated by these integers, namely
\[\left\langle a_1, \ldots, a_n\right\rangle = \left\{\sum_{i=1}^n c_ia_i\;:\;c_i\in \mathbb{Z}_{>0}\right\}.\]
The gcd condition ensures that this is a numerical semigroup.
\begin{theorem}\label{numerical_semigroup}
	Let $C$ be a plane curve branch. Then $S = S(C)$ is a numerical semigroup. In fact
	\[|\mathbb{Z}_{>0}\setminus S| = \dim_k k[[t]]/A.\]
\end{theorem}
\begin{proof}
	Let $C = \spec A$ and pick a normalisation $A\subseteq k[[t]]$ and let $v:k[[t]]\setminus\{0\}\to \mathbb{Z}_{\ge 0}$ be the valuation.
	
	It's clear that $S$ is closed under addition since if $v(a) = m$ and $v(b) = n$ then $v(ab) = m+n$.
	
	For the finite complement, one sees that the images of the elements $\{t^m:m\in \mathbb{Z}_{>0}\setminus S\}\subseteq k[[t]]$  in $k[[t]]/A$ form a $k$-basis. From this the claimed formula for $|\mathbb{Z}_{>0}\setminus S|$ follows, and also its finiteness as $\dim_k k[[t]]/A$ is finite by Theorem \ref{branch_normalisation}.
\end{proof}
From this one easily sees that $C$ is smooth if and only if $S(C) = \mathbb{Z}_{>0}$.

\begin{example}\label{example_semigroup}
	Let $A = k[[x,y]]/(x^m-y^n)$ for $m,n$ coprime positive integers and $C = \spec A$. We saw in Example \ref{example_parametrisation} that $x\mapsto t^n, y\mapsto t^m$ is a parametrisation of $A$, and hence $A\cong k[[t^m, t^n]]\subseteq k[[t]]$. So the valuation semigroup of $C$ is 
	\[S(C) = \left\langle m,n\right\rangle.\]
\end{example}

We also have some related invariants:
\begin{definition}
	Let $C$ be a plane curve branch. Its \emph{multiplicity} is $\mu(C) = \min S(C)$ the smallest element of its valuation semigroup. Its \emph{$\delta$-invariant} is $\delta(C) = |\mathbb{Z}_{>0}\setminus S(C)|$ the number of missing values in the valuation semigroup.
\end{definition}
These are in fact the same as multiplicity and the $\delta$-invariant defined in the usual way:
\begin{theorem}\label{welldef_delta_mult}
	Let $C = \spec A$ be a plane curve branch. Let $\mathfrak{m}\subseteq A$ be the maximal ideal of $A$.
	\begin{enumerate}
		\item We have that $\mu(C) = \dim_k \mathfrak{m}^n/\mathfrak{m}^{n+1}$ for all sufficiently large $n$. Also if $A = k[[x,y]]/(f)$ for some power series $f\in k[[x,y]]$, then $\mu(C)$ is the order of $f$.
		\item We have that $\delta(C) = \dim_k k[[t]]/A$ where we regard $A\subseteq k[[t]]$ as a subring by normalisation.
	\end{enumerate}
\end{theorem}
\begin{remark}
	Part (1) of the above theorem says that the multiplicity defined in terms of the valuation semigroup is equal to the multiplicity defined in terms of the Hilbert series, or as the order of vanishing of an equation defining the curve branch in the plane. Part (2) says that the $\delta$-invariant defined in terms of the valuation semigroup is equal to the $\delta$-invariant in the usual sense.
\end{remark}
\begin{proof}
	(1): Let $A\subseteq k[[t]]$ by normalisation, and let $\mu = \mu(C)$. For any $n$, we see that
	\[\mathfrak{m}^n = \left\{g(t) = \sum_{i\ge n\mu}^\infty a_it^i\;:\;f(t)\in A\right\},\]
	i.e. $\mathfrak{m}^n$ contains all power series of order $\ge n\mu$ in $A$. For sufficiently large $n$, $A$ will contain power series $g_{n\mu}, g_{n\mu+1},\ldots, g_{n\mu + \mu -1}$ such that $g_i$ has order $i$. It's easy to see then that these $g_i$ form a $k$-basis of $\mathfrak{m}^n/\mathfrak{m}^{n+1}$, giving the first claim.
	
	The second statement follows from the first one and \cite[Theorem 3.2(1)]{flenner1993multiplicities}.
	
	(2): This is clear from Theorem \ref{numerical_semigroup}.
\end{proof}
\subsection{Reducible germs}
In this subsection we look at reducible plane curve germs, and define their intersection numbers and tangency.
\begin{definition}\label{definition_intersection}
	Let $C = \spec A$ be a plane curve germ with branches $B_1, \ldots, B_n$ corresponding to minimal primes $\mathfrak{p}_1, \ldots, \mathfrak{p}_n$ of $A$. The \emph{local intersection number} of branches $B_i\not=B_j$ is given by
	\[i_C(B_i, B_j) = \dim_k A/(\mathfrak{p}_i + \mathfrak{p}_j).\]
\end{definition}
Note that the local intersection number is finite and always a positive integer. Indeed $\mathfrak{p}_i$ and $\mathfrak{p}_j$ are different height 0 primes, so their sum has height $\ge 1$. Since $A$ is 1-dimensional this means that $\mathfrak{p}_i + \mathfrak{p}_j$ is $\mathfrak{m}$-primary for $\mathfrak{m}$ the maximal ideal, so the above dimension is finite. It is positive since $\mathfrak{p}_i+\mathfrak{p}_j\subseteq \mathfrak{m}$.

If we write $A = k[[x,y]]/(f)$, then the branches $B_i, B_j$ correspond to distinct irreducible factors $f_i, f_j$ of $f$, and
\[i_C(B_i, B_j) = \dim_k k[[x,y]]/(f_i,f_j).\]

The following proposition states that in some sense, the intersection number of two branches can be computed in any suitable plane curve germ containing the given branches:
\begin{proposition}\label{intersection_independence}
	Let $C_1\into C_2$ be a closed embedding of plane curve germs, and suppose that $B,B'$ are distinct branches of $C_1$. Then we may also regard $B,B'$ as branches of $C_2$ via the above embedding, and we have
	\[i_{C_1}(B,B') = i_{C_2}(B,B').\]
	In particular if $C$ is a plane curve germ and $B\not=B'$ are branches of $C$, then
	\[i_C(B,B') = i_{C_{red}}(B,B').\]
\end{proposition}
\begin{proof}
	Let $C_i = \spec A_i$ for $i=1,2$ so the embedding corresponds to a surjection $A_2\onto A_1$. Pick a surjection $k[[x,y]]\onto A_2$, which also induces a surjection $k[[x,y]]\onto A_1$ via the embedding. This allows us to write
	\[A_1 = k[[x,y]]/(g_1),\qquad A_2 = k[[x,y]]/(g_2)\]
	for power series $g_1,g_2\in k[[x,y]]$ with $g_1\mid g_2$. The branches $B_1, B_2$ correspond to irreducible factors $f,f'$ of $g_1$ (and hence are also irreducible factors of $g_2$) and we have
	\[i_{C_1}(B,B') = \dim_k k[[x,y]]/(f,f') = i_{C_2}(B,B')\]
	which proves the claim.
\end{proof}
We also define tangency of branches. To do so we first need to have a way to compare the `tangent direction' to two distinct branches:
\begin{definition}
	Let $C = \spec A$ be a plane curve germ  and $\mathfrak{m}\subseteq A$ be the maximal ideal. Also suppose $A$ has at least two distinct branches (so that the cotangent space $\mathfrak{m}/\mathfrak{m}^2$ is 2-dimensional as $A$ cannot be a regular local ring). If $B = \spec A'$ is a branch of $C$ and $\mathfrak{m'}\subseteq A'$ is the maximal ideal, define the \emph{cotangent line} $L_B\subseteq \mathfrak{m}/\mathfrak{m}^2$ as follows:
	\begin{enumerate}
		\item If $B$ is smooth, $L_B$ is the kernel of the projection
		\[\mathfrak{m}/\mathfrak{m}^2\onto \mathfrak{m'}/\mathfrak{m'}^2.\]
		\item If $B$ is singular, $L_B$ is the preimage under the projection
		\[\mathfrak{m}/\mathfrak{m}^2\onto \mathfrak{m'}/\mathfrak{m'}^2\]
		of the unique line $L\subseteq \mathfrak{m'}/\mathfrak{m'}^2$ consisting of nilpotent elements in the associated graded ring $\gr_{\mathfrak{m}'}A'$ (which exists by Proposition \ref{cotangent_line}).
	\end{enumerate}
\end{definition}
Note this is well-defined: if $B$ is smooth then $\mathfrak{m'}/\mathfrak{m'}^2$ is 1-dimensional and $L_B$ is a line, and if $B$ is singular then the projection is an isomorphism (both cotangent spaces are 2-dimensional) so $L_B$ is again a line.

We define two branches to be tangent if their cotangent line is the same:
\begin{definition}\label{definition_tangency}
	Let $C = \spec A$ be a plane curve germ with at least two branches, and let $B,B'$ be branches of $C$. We say $B$ and $B'$ are \emph{tangent} if $L_B = L_{B'}$. If $C$ has a unique branch, we consider that branch tangent to itself.
	
	It's clear that tangency is an equivalence relation on branches of any plane curve germ $C$, and we call an equivalence class under tangency a \emph{tangency class}.
\end{definition}
Also define the following:
\begin{definition}\label{scheme_of_branches}
	Let $C = \spec A$ be a plane curve germ, and $T$ a tangency class of branches of $C$. We let $C_T$ be the reduced closed subscheme supported on the generic points of the branches in $T$ along with the closed point of $C$.
\end{definition}
Note that this is well-defined (i.e. the set consisting of the generic points of branches in $T$ and the closed point of $C$ is closed in $C$, so there is a unique reduced closed subscheme of $C$ with this set as its set of points). Moreover $C_T$ is a plane curve germ: if $A = k[[x,y]]/(f)$ and the branches in $T$ correspond to irreducible factors $h_1, \ldots, h_r$ of $f$, then we see easily that $C_T = \spec k[[x,y]]/(h_1h_2\ldots h_r)$.

We can characterise tangency in the following simpler way in terms of power series as well. Namely two branches given by power series are tangent if and only if their initial forms are powers of the same linear form.
\begin{proposition}\label{tangency_initial_form}
	Let $C = \spec A$ be a plane curve germ and write $A = k[[x,y]]/(f)$ for a power series $f$. Suppose $h\not=h'$ are distinct irreducible factors of $f$, corresponding to branches $B\not=B'$ of $C$. Then $B$ and $B'$ are tangent if and only if there is a linear form $l$ such that the initial forms of both $h$ and $h'$ are powers of $l$.
\end{proposition}
\begin{proof}
	We have as before that the cotangent space of $C$ is isomorphic to $(x,y)/(x,y)^2$. Under this isomorphism, the cotangent line of the branch $B$ corresponding to $h$ is exactly the line generated by $l$, where the initial form is a power of $l$. A similar statement holds for the branch $B'$, which implies the claim.
\end{proof}

We have the following property of the intersection pairing as it relates to tangency:
\begin{proposition}\label{intersection_tangency}
	Let $C$ be a plane curve germ, and $B\not= B'$ two distinct branches of $C$. Then
	\[i_C(B, B')\ge \mu(B)\mu(B')\]
	and the inequality is strict if and only if $B$ and $B'$ are tangent.
\end{proposition}
\begin{proof}
	This follows from \cite[Proposition 3.8(v)]{plane_algebraic_curves}. Note that in loc.cit. the intersection pairing is defined differently, but agrees with ours by Theorem 3.14 therein.
\end{proof}
Combining this with Proposition \ref{intersection_independence}, we also obtain that tangency of two given branches can be determined in any appropriate germ: if $C_1\into C_2$ is a closed embedding of plane curve germs, and $B\not=B'$ are branches of $C_1$, then they are tangent in $C_1$ if and only if they are tangent in $C_2$.
\subsection{Blowups of plane curve germs}
In this section we consider blowups of plane curve germs and how the invariants defined so far change under them.

The following theorem gives the general description of blowups of plane curve germs:
\begin{theorem}\label{blowup_germ}
	Let $C$ be a plane curve germ and consider the blowup $\widetilde{C}\to C$ in the closed point. Then
	\begin{enumerate}
		\item If $C$ is a plane curve branch, so is $\widetilde{C}$. More precisely if $C = \spec k[[x,y]]/(f)$ for some $y$-aligned $f$ of order $\mu$, then $\widetilde{C} = \spec k[[x,y_1]]/(f_1)$ where $f_1\in k[[x,y_1]]$ is the unique power series satisfying $f(x, xy_1) = x^\mu f_1(x,y_1)$. Moreover if $k[[x,y]]/(f)\into k[[t]], x\mapsto \varphi(t), y\mapsto \psi(t)$ is any parametrisation of $C$, then there is a commutative diagram
		\[\begin{tikzcd}
			{k[[x,y]]/(f)} && \\
			&& {k[[t]]} \\
			{k[[x,y_1]]/(f_1)}
			\arrow["\begin{array}{c} \substack{x\mapsto \varphi(t)\\ y \mapsto \psi(t)} \end{array}", from=1-1, to=2-3]
			\arrow["\begin{array}{c} \substack{x\mapsto x\\y\mapsto xy_1} \end{array}"', from=1-1, to=3-1]
			\arrow["\begin{array}{c} \substack{x\mapsto \varphi(t)\\y_1\mapsto \psi(t)/\varphi(t)} \end{array}"', from=3-1, to=2-3]
		\end{tikzcd}\]
		and the bottom map is a parametrisation of $\widetilde{C}$.
		\item If $C$ has a unique tangency class $T$ of branches, then $\widetilde{C}$ is a plane curve germ whose branches are exactly the strict transforms $\widetilde{B}$ of branches $B$ of $C$. More precisely if $C = \spec k[[x,y]]/(f)$ and all irreducible factors of $f$ are $y$-aligned and $f$ has order $\mu$, then
		\[\widetilde{C} = \spec k[[x,y_1]]/(f_1)\]
		where $f_1\in k[[x,y_1]]$ is the unique power series satisfying $f(x, xy_1) = x^\mu f_1(x,y_1)$.
		\item For general $C$ the blowup $\widetilde{C}$ is a disjoint union of schemes
		\[\widetilde{C} = \coprod_{T\in \mathcal{T}} \widetilde{C}_T\]
		indexed by the set $\mathcal{T}$ of tangency classes of branches of $C$. For each $T\in\mathcal{T}$, the reduced subscheme $\widetilde{C}_{T,red}$ is a plane curve germ isomorphic to the blowup of $C_T$ (here $C_T$ is the closed subscheme of $C$ containing the branches in $T$ as in Definition \ref{scheme_of_branches}). If $B$ is a branch of $C$ and $B\in T$, then the strict transform $\widetilde{B}$ is a branch of $\widetilde{C}_{T,red}$.
	\end{enumerate}
\end{theorem}
\begin{proof}
	By definition the blowup of $C = \spec A$ in its closed point is
	\[\widetilde{C} = \proj\left(\bigoplus_{n\ge 0} \mathfrak{m}^n\right)\]
	where $\mathfrak{m}\subseteq A$ is the maximal ideal. This scheme contains the generic points of each branch $B$ of $C$ (or equivalently of their strict transform $\widetilde{B}$) as these are outside of the closed point. The rest of the points lie on the exceptional divisor, given by tensoring with $k = A/\mathfrak{m}$:
	\[E = \proj\left(\bigoplus_{n\ge 0}\mathfrak{m}^n/\mathfrak{m}^{n+1}\right) = \proj \gr_{\mathfrak{m}}A.\]
	Write $A = k[[x,y]]/(f)$. By the same argument as in the proof of Proposition \ref{cotangent_line}, we have that
	\[\gr_{\mathfrak{m}}A\cong k[x,y]/(f_{in}).\]
	Where $f_{in}$ is the initial form of $f$. Taking $\proj$, we obtain that $E$ is a scheme with finitely many closed points corresponding to maximal homogeneous ideals in $k[x,y]/(f_{in})$ not containing both $x$ and $y$, which are easily seen to correspond to linear forms dividing $f_{in}$. These in turn correspond to tangency classes of branches of $C$ by Proposition \ref{tangency_initial_form}. If $T$ is a tangency class of branches of $C$, let $p_T\in E$ be the corresponding closed point (given by the vanishing of the linear form $l$ such that every branch in the tangency class has initial form a power of $l$). It's easy to see that if $B$ is a branch in the class $T$, then the strict transform $\widetilde{B}$ goes through the point $p_T$ and no other closed point.
	
	Hence $\widetilde{C}$ contains a closed point $p_T$ for each tangency class $T$, and for each branch $B$ of $C$ the generic point of the strict transform $\widetilde{B}$, each of which specialises to a unique closed point $p_T$. So $\widetilde{C}$ is the disjoint union of the local schemes $\widetilde{C}_T$, each one containing the unique closed point $p_T$ and the generic points of branches $\widetilde{B}$ with $B\in T$. This shows the decomposition claim in (3), and in the setups of (1) and (2) it follows that $\widetilde{C}$ is irreducible.
	
	Now suppose that $C$ contains a unique tangency class (this encompasses the situations in both (1) and (2)). We can write $A = k[[x,y]]/(f)$ so that every irreducible factor of $f$ is $y$-aligned (this is the case as soon as one irreducible factor is $y$-aligned by Proposition \ref{cotangent_line}). Then we can write $f$ as
	\[f = y^\mu + a_1(x)y^{\mu-1}+\ldots + a_{\mu-1}(x)y + a_\mu(x)\]
	with $a_i(x)$ of order $\ge i + 1$. If we substitute $y = xy_1$, every term becomes divisible by $x^\mu$, so there is indeed a unique power series $f_1(x,y_1)\in k[[x,y_1]]$ such that
	\[f(x, xy_1) = x^\mu f_1(x, y_1).\]
	The form for $f$ also implies that, in the Rees algebra
	\[R = \bigoplus_{n\ge 0}\mathfrak{m}^n,\]
	we have $y\in \sqrt{(x)}$ (considering both $x,y$ as degree 1 elements), meaning that the radical of the ideal generated by $x$ is the irrelevant ideal. This means that every prime ideal containing $x$ contains the irrelevant ideal, and hence
	\[\widetilde{C} = \proj R \cong \spec R_{(x)}\]
	where $R_{(x)}$ is the ring of degree 0 elements in the localisation $R_x$. By \cite[Lemma 8.1.2(e)]{liu} we have
	\[R_{(x)}\cong \frac{A[T]/(xT-y)}{\text{$x^\infty$-torsion}}.\]
	Using this, one can show that in fact $R_{(x)}\cong k[[x,y_1]]/(f_1)$: we do this by constructing maps in both directions.
	\begin{itemize}
		\item There is an obvious map $A = k[[x,y]]/(f)\to k[[x,y_1]]/(f_1)$ by sending $x\mapsto x, y\mapsto xy_1$. The map is well-defined as it sends $f$ to $f(x,xy_1) = x^\mu f_1$. This obviously extends to a map $A[T]/(xT-y)\to k[[x,y_1]]/(f_1)$ by sending $T\mapsto y_1$. It's easy to see that the right-hand ring has no $x$-torsion, so the map descends to a morphism
		\[R_{(x)} = \frac{A[T]/(xT-y)}{\text{$x^\infty$-torsion}}\to k[[x,y_1]]/(f_1)\]
		\item In the other direction, there is an obvious map $k[x,y_1]\to R_{(x)}$ via $x\mapsto x, y_1\mapsto T$. Because the blowup $\widetilde{C}\to C$ is projective and quasi-finite (i.e. has finite fibres, shown above), it is in fact finite by Zariski's Main Theorem \cite[Corollary 4.4.7]{liu}, so $R_{(x)}$ is finite over $A$. We also know that $R_{(x)}$ is a local ring (we showed above that $\widetilde{C}$ has a unique closed point in our case), therefore $R_{(x)}$ is a complete local ring (as it's finite over a complete local ring, see \cite[Exercise 4.3.17]{liu}). The maximal ideal is in fact generated by $x$ and $T$, so the above morphism extends to a map
		\[k[[x,y_1]]\to R_{(x)}.\]
		This map sends $f_1$ to 0 since it sends $x^\mu f_1 = f(x, xy_1)$ to $f = 0$ in $R_{(x)}$ and $R_{(x)}$ has no $x$-torsion, and hence descends to a map
		\[k[[x,y_1]]\to R_{(x)}.\] 
	\end{itemize}
	It's easy to check that the above maps are inverse to each other, and we obtain the desired isomorphism. We've hence shown that
	\[\widetilde{C}\cong \spec k[[x,y_1]]/(f_1).\]
	This shows that $\widetilde{C}$ is a plane curve germ of the desired form in (2), and in case (1), if $f$ is irreducible then so is $f_1$ by \cite[Lemma 1.4]{plane_algebraic_curves} so $\widetilde{C}$ is a plane curve branch of the desired form.
	
	For the parametrisation claim in (1), it's easy to see that the given maps are well-defined and they commute ($\psi(t)/\varphi(t)$ is a power series in $t$ because $\psi$ has higher order by Proposition \ref{parametrisation}). The image of $k[[x,y_1]]/(f_1)$ in $k[[t]]$ contains the image of $k[[x,y]]/(f)$, so the quotient is finite-dimensional and therefore gives a parametrisation of $k[[x,y_1]]/(f_1)$.
	
	For (3), we've already shown the desired decomposition of $\widetilde{C}$. We have that $C_T$ is reduced and all its branches are tangent: then the strict transform of $C_T$ in $\widetilde{C}$ is a reduced plane curve germ. This strict transform must contain the closed point $p_T$ and the generic point of each branch in $T$, hence it must in fact be $\widetilde{C}_{T, red}$ (the strict transform is a reduced closed subscheme containing exactly the points of $\widetilde{C}_T$). The statement on strict transforms was also shown above.
\end{proof}
\begin{remark}
	One can in fact prove a more refined version of (3): namely each $\widetilde{C}_T$ is itself a plane curve germ, by giving each branch the same `multiplicity' as in $C$. We shall not need this in what follows.
\end{remark}

We can now discuss how the valuation semigroup changes under blowups. To do this, we need to introduce the notion of an Apéry set of a numerical semigroup:
\begin{definition}
	Let $S$ be a numerical semigroup and $m\ge 1$ an integer. The \emph{$m$-Apéry set} of $S$ is the ordered set of integers $0\le a_0<a_1<\ldots<a_{m-1}$ such that
	\[S \cup \{0\} = \bigcup_{i=0}^{m-1}\left(a_i + m\mathbb{Z}_{\ge 0}\right).\]
\end{definition}
This means that each $a_i$ is the smallest element of $S\cup\{0\}$ in its residue class mod $m$, and we list the $a_i$ in increasing order. Note that we get $m$ numbers in this way, because $S$ contains all but finitely many positive integers.

Using this we can characterise the semigroup of a plane curve branch after a blowup. This is originally due to Apéry \cite{Apery}.
\begin{theorem}\label{blowup_semigroup}
	Let $C$ be a plane curve branch of multiplicity $\mu = \mu(C)$, and let $\pi: \widetilde{C}\to C$ be the blowup of $C$ in its closed point. Then $\widetilde{C}$ is also a plane curve branch, and if $S(\widetilde{C})$ has $\mu$-Apéry sequence $0\le a_0<a_1<\ldots<a_{\mu-1}$, then $S(C)$ has $\mu$-Apéry sequence $0\le a_0<a_1 + \mu <\ldots < a_i + i\mu <\ldots <a_{\mu-1}+(\mu-1)\mu$.
\end{theorem}
\begin{proof}
	It follows from Theorem \ref{blowup_germ} that $\widetilde{C}$ is a plane curve branch. The statement about the Apéry sets is \cite[Theorem 4.1]{poski2014plane}.
\end{proof}
\begin{remark}\label{semigroup_recovery}
	This theorem means that $S(C)$ determines $S(\widetilde{C})$ but the converse is more subtle. In fact $S(\widetilde{C})$ doesn't determine $S(C)$: one can show using the theorem above that the germs
	\[C_1 = \spec k[[t]]\qquad C_2 = \spec k[[x,y]]/(y^2-x^3)\qquad C_3 = \spec k[[x,y]]/(y^3-x^4)\]
	all have smooth blowups, but have semigroups $S_1 = \mathbb{Z}_{>0}, S_2 = \left\langle 2,3\right\rangle, S_3 = \left\langle 3,4\right\rangle$ respectively (see Example \ref{example_semigroup}), which are all different.
	
	However knowing $S(\widetilde{C})$ and $\mu(C)$ allows us to recover $S(C)$, by the above theorem. Indeed one sees that $S_1, S_2, S_3$ above all have different multiplicities.
\end{remark}
\begin{remark}\label{exceptional_semigroup}
	Theorem \ref{blowup_semigroup} also allows us to show that not every numerical semigroup is the valuation semigroup of a plane curve branch. Indeed, let
	\[S = \left\langle 3,4,5\right\rangle = \{3,4,5,6,\ldots\}\]
	be the set of integers $\ge 3$. This is a numerical semigroup with $\mu = 3$ and 3-Apéry set $0<4<5$. If we assume $S = S(C)$ for a plane curve branch $C$, then $S(\widetilde{C})$ would have 3-Apéry set $0\le a_0<a_1<a_2$ satisfying
	\[a_0 = 0, a_1 + 3 = 4, a_2 + 6 = 5\]
	which is clearly not possible. Hence $S$ above is not the valuation semigroup of a plane curve branch.
	
	This shows that not every numerical semigroup arises as the valuation semigroup of a plane curve branch. It is natural to ask whether one can characterise numerical semigroups which do arise in this way. If $k$ has characteristic 0, such a characterisation is possible, see \cite[Appendix, Proposition 3.2.1]{zariski2006moduli} if $k = \mathbb{C}$ and \cite{semigroup_branches} in general. It is unknown to the author whether there is a similar characterisation in positive characteristic.
\end{remark}
We now determine the intersection numbers of branches after a blowup:
\begin{theorem}\label{blowup_intersection}
	Let $C$ be a plane curve germ and $B, B'$ branches of $C$. Consider the blowup $\widetilde{C}\to C$ at the closed point. Then
	\begin{enumerate}
		\item The strict transforms $\widetilde{B}, \widetilde{B'}$ meet at a closed point on the exceptional divisor if and only if they are tangent.
		\item If $B,B'$ are tangent and $\widetilde{B}, \widetilde{B'}$ meet at a point $z$, then they are branches of the plane curve germ $\widetilde{C}_z = (\spec \mathcal{O}_{\widetilde{C}, z})_{red}$ with closed point $z$, and
		\[i_{\widetilde{C}_z}(\widetilde{B}, \widetilde{B'}) = i_C(B, B')-\mu(B)\mu(B').\]
	\end{enumerate}
\end{theorem}
\begin{remark}
	Note that $B,B'$ being tangent implies that $i_C(B,B')>\mu(B)\mu(B')$ by Proposition \ref{intersection_tangency}, so the intersection number above is positive.
\end{remark}
\begin{proof}
	(1) follows from Theorem \ref{blowup_germ}(3).
	
	For (2), we may assume that $C$ only has one tangency class again by Theorem \ref{blowup_germ}(3). Write $C = \spec k[[x,y]]/(f)$ where each irreducible factor of $f$ is $y$-aligned. Let $h,h'$ be the irreducible factors of $f$ corresponding to the branches $B, B'$ respectively.
	
	By Theorem \ref{blowup_germ}(2) we have that
	\[\widetilde{C}_z = \spec k[[x,y_1]]/(f_1)\]
	is a plane curve germ, and the strict transforms $\widetilde{B}, \widetilde{B'}$ are branches of it. Here $f(x, xy_1) = x^{\mu_f} f_1$, $\mu_f$ being the order of $f$. We can in fact write
	\[\widetilde{B} = \spec k[[x,y_1]]/(h_1),\qquad \widetilde{B'}=\spec k[[x,y_1]]/(h_1')\]
	where $h(x,xy_1)=x^{\mu}h_1$ and $h'(x, xy_1) = x^{\mu'}h_1'$, with $\mu = \mu(B)$ and $\mu' = \mu(B')$ being the orders of $h,h'$ respectively. It's clear that $h_1, h_1'$ are irreducible factors of $f_1$.
	
	Pick a parametrisation $k[[x,y]]/(h)\into k[[t]], x\mapsto \varphi(t), y\mapsto \psi(t)$. We have the commutative diagram of parametrisations
	\[\begin{tikzcd}
		{k[[x,y]]/(h)} && \\
		&& {k[[t]]} \\
		{k[[x,y_1]]/(h_1)}
		\arrow["\begin{array}{c} \substack{x\mapsto \varphi(t)\\ y \mapsto \psi(t)} \end{array}", from=1-1, to=2-3]
		\arrow["\begin{array}{c} \substack{x\mapsto x\\y\mapsto xy_1} \end{array}"', from=1-1, to=3-1]
		\arrow["\begin{array}{c} \substack{x\mapsto \varphi(t)\\y_1\mapsto \psi(t)/\varphi(t)} \end{array}"', from=3-1, to=2-3]
	\end{tikzcd}\]
	and we have, by \cite[Theorem 3.14]{plane_algebraic_curves} that
	\begin{align*}
		i_C(B, B') &= \ord_t\, h'(\varphi(t), \psi(t))\\
		i_{\widetilde{C}_z}(\widetilde{B}, \widetilde{B'}) &= \ord_t\, h_1'(\varphi(t), \psi(t)/\varphi(t)).
	\end{align*}
By definition of $h_1'$ we have that 
\[x^{\mu'}h_1'(x, y_1) = h_1(x, xy_1)\]
which implies that
\[h'(\varphi(t), \psi(t)) = \varphi(t)^{\mu'}h_1'(\varphi(t), \psi(t)/\varphi(t)),\]
and taking orders in $t$ gives that
\begin{align*}
	i_C(B, B') &= \mu\mu' + i_{\widetilde{C}_z}(\widetilde{B}, \widetilde{B'})\\
	&=\mu(B)\mu(B') + i_{\widetilde{C}_z}(\widetilde{B}, \widetilde{B'})
\end{align*}
where we used that $\ord_t\,\varphi(t) = \mu$ by Proposition \ref{parametrisation}.
\end{proof}
\section{Global curves}\label{section_global_curves}
Let $k$ be an algebraically closed field. In this section we look at projective curves over $k$ which are locally planar, and show how certain quantities associated to them can be recovered in terms of their local germs. We do this for the multiplicity (Theorem \ref{local_global_mult}), $\delta$-invariant (Theorem \ref{local_global_delta}) and intersection numbers (Theorem \ref{local_global_intersection}). We will also give a classification of singularities with small $\delta$-invariant and singular curves of small genus (Proposition \ref{prop_small_genus_singularities}, Theorem \ref{singularity_classification}).

For us a \emph{projective curve} over $k$ is just a connected projective $k$-scheme $Y\to \spec k$ which is pure of dimension 1. In particular it can be reducible and non-reduced.
\begin{definition}
	Let $Y$ be a projective curve over $k$. We say $Y$ is \emph{locally planar} if for any closed point $p\in Y$, the completed local ring $\widehat{\mathcal{O}}_{Y,p}$ is a quotient of $k[[x,y]]$. If $Y$ is locally planar and $p\in Y$ is a closed point, the \emph{germ of $Y$ at $p$} is
	\[Y_p = \spec \widehat{\mathcal{O}}_{Y,p}\]
	which is a plane curve germ. A \emph{branch of $Y$ through $p$} is a branch of $Y_p$.
\end{definition}
Our motivation for studying these is the fact that if $\mathcal{O}_K$ is a discrete valuation ring with residue field $k$, then the special fibre of an arithmetic surface over $\mathcal{O}_K$ is a locally planar projective curve over $k$ (Proposition \ref{special_fibre_planar}).
\subsection{Relating global and local invariants}
If $Y$ is a locally planar projective curve, and $p\in Y$ a closed point, then by the previous section we can define the multiplicity and $\delta$-invariant of each branch of $Y$ through $p$, and the intersection number of distinct branches. In this subsection we will show how to recover the multiplicity and $\delta$-invariant of $Y$ at its points $p$, and the local intersection numbers between the components of $Y$.

We start with the multiplicity. Let $Y$ be a locally planar projective curve over $k$. If $p\in Y$ is a closed point, we can define the multiplicity $\mu_p(Y)$ of $Y$ at $p$, it is given as follows:
\begin{enumerate}
	\item If $A = \mathcal{O}_{Y,p}$ with maximal ideal $\mathfrak{m}$, then $\mu_p(Y) = \dim_k \mathfrak{m}^n/\mathfrak{m}^{n+1}$ for all sufficiently large $n$ (this is the definition in terms of Hilbert series). If $\widehat{A}$ is the completion of $A$ with maximal ideal $\widehat{\mathfrak{m}}$ then $\mathfrak{m}^n/\mathfrak{m}^{n+1}\cong \widehat{\mathfrak{m}}^n/\widehat{\mathfrak{m}}^{n+1}$ so we can replace the local ring by its completion.
	\item If $Y$ has an open subset isomorphic to $\spec k[x,y]/(f)$ for some $f\in k[x,y]$, and $p$ corresponds to the origin, then $\mu_p(Y)$ is the order of vanishing of $f$ at the origin.
\end{enumerate}
The equivalence of these two definitions follows from \cite[Theorem 3.2(1)]{flenner1993multiplicities}.

We can recover this (for a reduced curve) from the multiplicities of the branches through $p$:
\begin{theorem}\label{local_global_mult}
	Let $Y$ be a reduced locally planar projective curve over $k$ and $p\in Y$ a closed point. Let $B_1, \ldots, B_n$ be the branches of $Y_p$. Then 
	\[\mu_p(Y) = \sum_{i=1}^n \mu(B_i).\]
\end{theorem} 
\begin{proof}
	Write $Y_p = \spec k[[x,y]]/(f)$ for $f\in k[[x,y]]$. Then the brances $B_i$ correspond to irreducible factors $f_i$ of $f$. Moreover, $\mu_p(Y)$ is the order of $f$, and $\mu(B_i)$ is the order of $f_i$, so the claim just follows from
	\[f = f_1f_2\ldots f_r\]
	which is true because $Y$ is reduced so $f$ has no repeated irreducible factors.
\end{proof}
\begin{remark}
	One can easily extend this to non-reduced curves by defining the multiplicity of each branch inside $Y_p$, but we will not need this.
\end{remark}

We turn to the $\delta$-invariant. The following properties are stated in \cite[§7.5]{liu}. For $Y$ an integral projective curve over $k$, write $\nu:\widetilde{Y}\to Y$ for the normalisation of $Y$. Then we have an exact sequence of sheaves
\[0\to \mathcal{O}_Y\to \nu_*\mathcal{O}_{\widetilde{Y}}\to Q\to 0\]
where $Q$ is a skyscraper sheaf supported at the singular points of $Y$. If $p$ is a closed point on $Y$, define the $\delta$-invariant of $\Gamma$ at $p$ to be
\[\delta_p(\Gamma) = \dim_k Q_p\]
which is 0 is $p$ is a smooth point of $\Gamma$, and a positive integer otherwise. 

We then have the relation
\[g(\widetilde{Y}) = p_a(Y) - \sum_p \delta_p(Y)\]
where the sum is over the closed points of $Y$. Here $g$ denotes the geometric genus and $p_a$ the arithmetic genus. Thus the $\delta$-invariant can be though of as the contribution of a singularity to the `drop' in the genus of the curve.

For any closed point $p\in Y$ we can express the $\delta$-invariant $\delta_p(Y)$ by means of the delta invariants and intersection numbers of the branches:
\begin{theorem}\label{local_global_delta}
	Let $Y$ be an integral locally planar projective curve over $k$ and let $p\in Y$ be a closed point. Suppose $Y_p$ has branches $B_1, \ldots, B_r$. Then
	\[\delta_p(Y) = \sum_{i=1}^r \delta(B_i) + \sum_{i<j}i_{Y_p}(B_i, B_j).\]
\end{theorem}
\begin{proof}
	Put $A = \widehat{\mathcal{O}}_{Y, p}$. By \cite[Lemma 33.41.2(4)+(9), Tag 0C1R]{stacks-project}, we have that 
	\[\delta_p(Y) = \dim_k A'/A\]
	where $A'$ is the integral closure of $A$ in its total ring of fractions. If $A$ has minimal primes $\mathfrak{p}_1, \ldots, \mathfrak{p}_r$ corresponding to the branches $B_1, \ldots, B_r$, then we have 
	\[A' = \bigoplus_{i=1}^r (A/\mathfrak{p}_i)'\]
	where $(A/\mathfrak{p}_i)'$ is the integral closure of $A/\mathfrak{p}_i$ in its field of fractions. We have a chain of inclusions
	\[A \subseteq \bigoplus_{i=1}^r A/\mathfrak{p}_i \subseteq \bigoplus_{i=1}^r (A/\mathfrak{p}_i)'\]
	and the latter has index
	\[\dim_k\left(\frac{\bigoplus_{i=1}^r (A/\mathfrak{p}_i)'}{\bigoplus_{i=1}^r A/\mathfrak{p}_i}\right) = \sum_{i=1}^r \dim_k\left(\frac{(A/\mathfrak{p}_i)'}{A/\mathfrak{p}_i}\right) = \sum_{i=1}^r \delta(B_i)\]
	by Theorem \ref{welldef_delta_mult}(2). So we just need to compute the index of the first inclusion. To do this we may lift along a surjection $k[[x,y]]\onto A$, and try to compute
	\[\dim_k\left(\coker\left(k[[x,y]]/(f_1f_2\ldots f_r)\to \bigoplus_{i=1}^r k[[x,y]]/(f_i)\right)\right)\]
	for distinct irreducible $f_1, \ldots, f_r\in k[[x,y]]$. By \cite[Lemma 10.1.50]{liu}, this dimension is equal to
	\[\sum_{1\le i <j \le r} \dim_k k[[x,y]]/(f_i,f_j)\]
	which proves the desired claim.
\end{proof}

Now we look at intersection numbers. Let $Y$ be a locally planar projective curve over $k$, $p\in Y$ be a closed point, and $\Gamma_1, \ldots, \Gamma_n$ be the irreducible components of $Y$ going through $p$.

Then the $\Gamma_i$ correspond to minimal primes $\mathfrak{p}_i$ of the local ring $A = \mathcal{O}_{Y,p}$. We can define an intersection number on distinct components $\Gamma_i\not=\Gamma_j$ through $p$ via
\[i_p(\Gamma_i, \Gamma_j) = \dim_k A/(\mathfrak{p}_i + \mathfrak{p}_j).\]
Again this is easily seen to be a positive integer. We also have that
\[i_p(\Gamma_i, \Gamma_j) = \dim_k \widehat{A}/(\widehat{\mathfrak{p}}_i + \widehat{\mathfrak{p}}_j)\]
where $\widehat{A}$ is the completion of $A$ and $\widehat{\mathfrak{p}} = \mathfrak{p}\widehat{A}$.

Note that each $\Gamma_{i, p}$ (the germ of $\Gamma_i$ at $p$) is a closed subscheme of $Y_p$, so it makes sense to take the intersection number of a branch of $\Gamma_{i,p}$ with a branch of $\Gamma_{j,p}$ (as they are both branches of $Y_p$).
\begin{theorem}\label{local_global_intersection}
	Let $Y$ be a locally planar projective curve , $p\in Y$ a closed point and $\Gamma, \Gamma'$ distinct components of $Y$ passing through $p$. Then
	\[i_p(\Gamma, \Gamma') = \sum_{B, B'}i_{Y_p}(B, B')\]
	where the sum runs over all branches $B$ of $\Gamma_p$ and $B'$ of $\Gamma'_p$.
\end{theorem}
\begin{proof}
	Let $A = \mathcal{O}_{Y,p}$, and let $\mathfrak{p}, \mathfrak{p}'$ be the minimal primes corresponding to $\Gamma, \Gamma'$ respectively. Branches of $\Gamma$ through $p$ correspond to minimal primes above $\widehat{\mathfrak{p}} = \mathfrak{p}\widehat{A}$, and we have that
	\[\widehat{\mathfrak{p}} = \bigcap \widetilde{\mathfrak{p}} = \prod \widetilde{\mathfrak{p}},\]
	with both the intersection and the product taken over minimal primes of $\widehat{\mathfrak{p}}$. The intersection is equal to the product because each ideal is principal, as $Y$ is locally planar. By lifting along a surjection $k[[x,y]]\onto \widehat{A}$ and using additivity of the intersection number \cite[Lemma 9.1.4(c)]{liu}, we obtain 
	\[i_p(\Gamma, \Gamma') = \dim_k \widehat{A}/(\widehat{\mathfrak{p}}_i + \widehat{\mathfrak{p}}_j) = \sum_{\widetilde{\mathfrak{p}}, \widetilde{\mathfrak{p}}'}\dim_k \widehat{A}/(\widetilde{\mathfrak{p}} + \widetilde{\mathfrak{p}}') = \sum_{B, B'}i_{Y_p}(B, B').\]
\end{proof}
\subsection{Classifying singular curves in small genus}
We will now classify singularities with small $\delta$-invariant, and singular curves of small genus.

First we define the following singularities:
\begin{definition}Let $C$ be a reduced plane curve germ. We say that $C$ is
	\begin{itemize}
		\item a \emph{cusp} if $C$ is a branch with semigroup $S = \left\langle 2,3\right\rangle$
		\item a \emph{node} if $C$ has two smooth branches intersecting with intersection number 1 (transversally)
		\item a \emph{$(2,5)$-cusp} if $C$ is a branch with semigroup $S = \left\langle 2,5\right\rangle$
		\item a \emph{tacnode} if $C$ has two smooth branches intersecting with intersection number 2.
	\end{itemize}
	If $Y$ is an integral locally planar curve and $p\in Y$ is a closed point such that $Y_p$ is a cusp, we also say that $p$ is a cusp of $Y$ (similarly for node, $(2,5)$-cusp and tacnode).
\end{definition}
We claim that these are all plane curve singularities of $\delta$-invariant $\le 2$:
\begin{proposition}\label{prop_small_genus_singularities}
	Let $Y$ be an integral locally planar projective curve, and $p\in Y$ a closed point. Then
	\begin{itemize}
		\item If $\delta_p(Y) = 1$ then $p$ is a cusp or a node of $Y$
		\item If $\delta_p(Y) = 2$ then $p$ is a $(2,5)$-cusp or a tacnode of $Y$. 
	\end{itemize}
\end{proposition}
\begin{proof}
	Suppose $Y_p$ has branches $B_1, \ldots, B_r$. By Theorem \ref{local_global_delta} we have that
	\[\delta_p(Y) = \sum_{i=1}^r \delta(B_i) + \sum_{i<j}i_{Y_p}(B_i, B_j).\tag{$\dagger$}\]
	Suppose that $\delta_p(Y) = 1$ or $2$. Then there can be at most 2 branches: indeed if there are at least 3 branches, then the sum of the intersection numbers on the right hand side of $(\dagger)$ is at least 3, a contradiction. 
	
	If there is a unique branch, then $Y_p$ is a plane curve branch. If $\delta(Y_p) = 1$ then it has semigroup $\left\langle 2,3\right\rangle$ and $p$ is a cusp of $Y$. If $\delta(Y_p)=2$ then it has semigroup $\left\langle 2,5\right\rangle$ (the only numerical semigroups missing exactly two numbers are $\left\langle 2,5\right\rangle$ and $\left\langle 3,4,5\right\rangle$, but the latter is not the semigroup of a plane curve branch by Remark \ref{exceptional_semigroup}), so $p$ is a $(2,5)$-cusp of $Y$.
	
	If there are two branches $B,B'$, they must both be smooth: otherwise $\delta(B)+\delta(B')\ge 1$ and $i_{Y_p}(B, B')\ge \mu(B)\mu(B')\ge 2$ (since at least one of them has multiplicity $\ge 2$), so the right-hand side of $(\dagger)$ is at least 3. Hence $Y_p$ has two smooth branches intersecting with multiplicity 1 or 2, i.e. $p$ is a node or a tacnode of $Y$.
\end{proof}
Using this we can classify locally planar curves of small genus:
\begin{theorem}\label{singularity_classification}
	Let $Y$ be an integral locally planar projective curve over $k$. Let $p = p_a(Y)$ be its arithmetic genus and $g = g(\widetilde{Y})$ be the genus of its normalisation. Then $0\le g \le p$ and
	\begin{enumerate}
		\item If $p=0$, then $g=0$ and $Y\cong \mathbb{P}^1$.
		\item If $p=1$, then either $g=1$ and $Y$ is smooth, or $g=0$ and $Y$ is a rational curve with a unique node or a unique cusp.
		\item If $p=2$, then either $g=2$ and $Y$ is smooth, $g=1$ and $Y$ has a unique node or a unique cusp, or $g=0$ and $Y$ is a rational curve with two cusps, two nodes, a node and a cusp, a unique $(2,5)$-cusp or a unique tacnode.
	\end{enumerate}
\end{theorem}
\begin{proof}
We have that 
\[p - g = \sum_{p\in Y}\delta_p(Y)\ge 0\]
where the sum is over singular points $p\in Y$. This shows $0\le g\le p$, and $Y$ is smooth if and only if the sum on the right is empty, i.e. if $p = g$. If $g = 0$ then $Y$ must be a rational curve.

If $p = 0$ then $g=0$ as well, and the sum on the right is empty, so $Y$ is smooth, so $Y\cong \mathbb{P}^1$ (remember $k$ is algebraically closed).

If $p = 1$ then either $g = 1$ and $Y$ is smooth, or $g=0$ and there is a unique singular point with $\delta = 1$, which is either a node or a cusp by Theorem \ref{singularity_classification}.

If $p=2$ then either $g=2$ and $Y$ is smooth, $g=1$ and there is a unique singular point with $\delta = 1$ which is a cusp or a node, or $g = 0$. In this last case we can have two singular points with $\delta = 1$ for both, so they are both cusps or nodes, or a unique singular point with $\delta = 2$, which is then a $(2,5)$-cusp or a tacnode by Theorem \ref{singularity_classification} again.
\end{proof}
\section{Arithmetic surfaces}\label{section_arithmetic_surfaces}
Let $\mathcal{O}_K$ be a DVR with fraction field $K$, uniformiser $\pi$ and residue field $k$. We assume $k$ is algebraically closed. Recall that an \emph{arithmetic surface} over $\mathcal{O}_K$ is a proper flat regular $\mathcal{O}_K$-scheme $ X\to \spec \mathcal{O}_K$ of relative dimension 1. Given a smooth projective curve $C$ over $K$, a \emph{regular model} of $C$ is an arithmetic surface $X$ over $\mathcal{O}_K$ along with an isomorphism $X_K\cong C$.

We will be interested in the special fibres $X_s$ of arithmetic surfaces, especially for regular models of curves. We will define a combinatorial object attached to the special fibre (the \emph{type}), and investigate how this object changes under blowups at closed points. The main result is Theorem \ref{blowup_arithmetic_surface}, which states that for a blowup $\widetilde{X}\to X$ of an arithmetic surface at a closed point, the type of $X$ determines the type of $\widetilde{X}$ and vice versa. We then specialise to models of curves and show that the type of the minimal regular model determines the type of the minimal regular normal crossings model and vice versa.

In fact the special fibre is locally planar:
\begin{proposition}\label{special_fibre_planar}
	Let $X$ be an arithmetic surface over $\mathcal{O}_K$. Then $X_s$ is a locally planar projective curve.
\end{proposition}
\begin{proof}
	Let $p\in X_s$ be a closed point, and $\Gamma_1, \ldots, \Gamma_r$ be the components of $X_s$ passing through $p$ of respective multiplicities $d_i$. Then by \cite[Exercise 9.2.7]{liu} there is a flat scheme $Z\to \spec \mathcal{O}_K$ of finite type and of relative dimension 2, with a closed point $z\in Z_s$ that is regular in $Z_s$, and an isomorphism
	\[\mathcal{O}_{Z,z}/(u_1^{d_1}\ldots u_r^{d_r}-\pi a)\cong \mathcal{O}_{X, p}\]
	where the image of $u_i\in \mathcal{O}_{Z,z}$ in $\mathcal{O}_{X,p}$ is a local equation for $\Gamma_i$, and $a\in \mathcal{O}_{Z,z}^\times$. Since $k$ is algebraically closed, $Z$ is in fact smooth at $z$, which implies that $\widehat{\mathcal{O}}_{Z,z}\cong \mathcal{O}_K[[x,y]]$ is a power series ring. This shows that $\widehat{\mathcal{O}}_{X_s, p}$ is a quotient of $\widehat{\mathcal{O}}_{Z_s, z}\cong k[[x,y]]$.
\end{proof}

This makes it possible to apply our previous results to the special fibre of an arithmetic surface.

If $X$ is an arithmetic surface over $\mathcal{O}_K$, then the special fibre $X_s$ is an effective Cartier divisor which decomposes as
\[X_s = \sum_\Gamma d_\Gamma \Gamma\]
where the sum is over the components $\Gamma$ of $X_s$, and $d_\Gamma$ is the multiplicity of $\Gamma$ in $X_s$ (defined as the length of the local ring $\mathcal{O}_{X_s, \xi_\Gamma}$ where $\xi_\Gamma$ is the generic point of $\Gamma$). The free abelian group on the components $\Gamma$ is identified with the group $\Div(X)$ of vertical divisors on $X$, and there is an intersection pairing
\[\cdot: \Div(X)\times \Div(X)\to \mathbb{Z}\]
which is a symmetric bilinear form. It is compatible with local intersection numbers in the sense that for components $\Gamma\not=\Gamma'$,
\[\Gamma\cdot \Gamma' = \sum_{p\in \Gamma\cap \Gamma'} i_p(\Gamma, \Gamma')\]
where $i_p(\Gamma, \Gamma')$ is the local intersection number, see \cite[Theorem 9.1.12]{liu}.
\subsection{The type of an arithmetic surface}
We first define an abstract type. This is meant to be a combinatorial description of a projective locally planar curve $Y$ over $k$, along with its singularities.
\begin{definition}\label{definition_type}
	An \emph{(abstract) type} is a tuple $T = (\comp, d, g, \sing, \inc, \br, S, i)$ where
	\begin{enumerate}
		\item $\comp$ is a nonempty finite set whose elements $\Gamma$ are called components, and $d:\comp \to \mathbb{Z}_{>0}$, $g:\comp \to \mathbb{Z}_{\ge 0}$ are functions, which we think of as associating to $\Gamma$ its multiplicity $d_\Gamma$ and geometric genus $g_\Gamma$.
		\item $\sing$ is a finite set whose elements $p$ are called singular points, along with an incidence relation $\inc\subseteq \sing\times\comp$; if $(p, \Gamma)\in \inc$ we say $p$ is on $\Gamma$ and write $p\in \Gamma$.
		\item $\br = (\br_{\Gamma, p})_{\Gamma,p}$ is a collection of nonempty finite sets $\br_{\Gamma, p}$ for all components $\Gamma$ and singular points $p$ with $p\in \Gamma$. Elements $B$ of $\br_{\Gamma, p}$ are  branches of $\Gamma$ through $p$. $S$ is a function associating to each branch $B\in  \coprod_{\substack{p, \Gamma\\p\in \Gamma}}\br_{\Gamma, p}$ a numerical semigroup $S(B)$ which we think of as the valuation semigroup of the branch $B$. $i$ is a function associating to each pair of distinct branches $B,B'$ through the same point $p$ an intersection number $i_p(B, B')$, which is a positive integer. (Note that we don't require $B$ and $B'$ to be branches of the same component, just that they are branches through the same point).
	\end{enumerate}
\end{definition}
There is an obvious notion of isomorphism of types (given by a bijection on components, singular points and branches in a compatible way and preserving multiplicities, genera, semigroups and intersection numbers).

Applying this to the case of arithmetic surfaces:
\begin{definition}\label{def_type_of_surface}
	Let $X$ be an arithmetic surface over $\mathcal{O}_K$. The \emph{type} of $X$ is the (isomorphism class of the) abstract type $(\comp, d, g, \sing, \inc, \br, S, i)$ given as follows:
	\begin{enumerate}
		\item $\comp$ is the set of irreducible components of $X_s$, and for each component $\Gamma$, $d_\Gamma$ is its multiplicity in $X_s$ and $g_\Gamma = g(\widetilde{\Gamma})$ is the genus of its normalisation.
		\item $\sing$ is the set of singular points of $(X_s)_{\mathrm{red}}$, and $\inc$ is the obvious incidence relation between singular points and components ($(p,\Gamma)\in \inc$ if and only if $p\in \Gamma$).
		\item For all components $\Gamma$ and singular points $p$ with $p\in \Gamma$, $\br_{\Gamma,p}$ is the set of branches of $\Gamma$ through $p$, $S$ associates to a branch $B$ its valuation semigroup $S(B)$, $i$ associates to two distinct branches $B, B'$ through a common point $p$ their intersection number $i_p(B, B') = i_{X_{s,p}}(B, B')$.
	\end{enumerate}
\end{definition}
We can easily see that this is well-defined (i.e. $\comp, \sing, \br_{\Gamma, p}$ are finite sets, $\comp, \br_{\Gamma,p}$ are nonempty, etc.) as $X_s$ is locally planar by Proposition \ref{special_fibre_planar}.

\begin{remark}
	Note that we haven't tried to give a notion of type such that every type is realised by some arithmetic surface. Indeed, there are obvious necessary conditions a type has to satisfy to be realisable this way: the special fibre must be connected, the intersection numbers must extend to a bilinear intersection pairing on the special fibre satisfying the adjunction formula, the intersection numbers must be compatible with the multiplicities (i.e. satisfy Theorem \ref{intersection_tangency}), the semigroups must arise as the numerical semigroup of a plane curve branch. It would be interesting to give a set of conditions on a type which are necessary and sufficient for the type to be realised in this way.
\end{remark}

Hence the type of an arithmetic surface $X$ keeps track of the following information about the special fibre:
\begin{itemize}
	\item The configuration of components and singular points,
	\item The multiplicity of each component,
	\item The geometric genus of each component,
	\item At each singular point, the set of branches along with their valuation semigroups and the intersection numbers between the branches.
\end{itemize}
In addition to this information, we can recover from the type the following information:
\begin{itemize}
	\item For a component $\Gamma$ and a singular point $p$ on it, the multiplicity $\mu_p(\Gamma)$. Indeed this follows from Theorem \ref{local_global_mult} as we know the multiplicities of each branch of $\Gamma$ through $p$, via their semigroups.
	\item For a component $\Gamma$ and a singular point $p$ on it, the $\delta$-invariant $\delta_p(\Gamma)$. This follows from Theorem \ref{local_global_delta} as we know the semigroups of the branches of $\Gamma$ through $p$ and their intersection numbers. Thus we also recover the arithmetic genus of each component $\Gamma$ via
	\[p_a(\Gamma) = g_\Gamma + \sum_{p\in \Gamma}\delta_p(\Gamma)\]
	summing over singular points $p\in \Gamma$.
	\item For components $\Gamma, \Gamma'$, the local intersection number $i_p(\Gamma, \Gamma')$, and hence the whole intersection pairing on vertical divisors. This follows from Theorem \ref{local_global_intersection} as we know the intersection numbers between branches.
\end{itemize}

In low genus one can often specify a type via a picture. On such a picture we draw a curve for each component, and label it with its multiplicity and genus (in blue in this paper, and only if multiplicity is $\ge 2$ or the genus is $\ge 1$). At intersection points, we draw a transversal intersection if the intersection number is 1, and a `tangent' intersection if it is 2 (these are the only intersection numbers that can arise in genus $\le 2$). Additionally we label each singular branch with its valuation semigroup (in red). See Example \ref{example:low_genus_singularities} for examples in genus 1 and 2, the Kodaira types in Section \ref{section_low_genus}, and Example \ref{example_genus_3} for similar genus 3 examples.

\subsection{Behaviour of the type under a blowup}
Now we investigate how the type of an arithmetic surface changes when we blow up at a singular point on the special fibre. First we establish some properties of blowups in this context, drawing on our results on blowups of plane curve germs.
\begin{proposition}\label{blowup_arithmetic_surface}
	Let $X$ be an arithmetic surface over $\mathcal{O}_K$ and $p\in X_s$ a closed point. Consider the blowup $f:\widetilde{X}\to X$ of $X$ at $p$, and let $E = f^{-1}(p)$ be the exceptional divisor. Then we have the following:
	\begin{enumerate}
		\item There is a correspondence
		\[\left\{\text{tangency classes of branches of $X_s$ through $p$}\right\}\leftrightarrow\left\{\substack{\text{closed points $q\in E$}\\\text{intersecting other components of $\widetilde{X}_s$}}\right\}\]
		Which sends the class of a branch $[B]$ to the point $q\in E$ where the strict transform $\widetilde{B}$ meets $E$.
		\item If $B$ is a branch through $p$ whose strict transform $\widetilde{B}$ meets $E$ at $q$, then $\widetilde{B}$ is a branch of $\widetilde{X}_s$ through $q$, and we have the intersection number
		\[i_{\widetilde{X}_{s, q}}(\widetilde{B}, E) = \mu(B).\]
		Note this makes sense as we can consider $E$ to be a branch of $\widetilde{X}_s$ through $q$.
		\item If $B,B'$ are tangent branches through $p$, then their strict transforms $\widetilde{B}, \widetilde{B'}$ are branches of $\widetilde{X}_s$ through the corresponding point $q\in E$, and we have
		\[i_{\widetilde{X}_{s,q}}(\widetilde{B}, \widetilde{B'}) = i_{X_{s, p}}(B, B') - \mu(B)\mu(B').\]
	\end{enumerate}  
\end{proposition}
\begin{proof}
	(1) The exceptional divisor $E$ meets other components of $\widetilde{X}_s$ precisely at points where the strict transforms of branches of $X_s$ through $p$ meet $E$. By Theorem \ref{blowup_germ}(3), these points are in bijection with tangency classes of branches of $X_s$ through $p$, and the bijection is as given above.
	
	(2) Let $C = X_{s,p}$. By Theorem \ref{blowup_germ} the blowup of $C_{red}$ has a component $\widetilde{C}_q$ with closed point $q$, and $\widetilde{B}$ is a branch of $\widetilde{C}_q$. It's easy to see that $\widetilde{C}_q$ is a closed subscheme of $\widetilde{X}_{s,q}$ and therefore $\widetilde{B}$ is a branch of $\widetilde{X}_s$ through $q$.
	
	If we write $B = \spec k[[x,y]]/(f)$ for $f$ a $y$-aligned power series of order $\mu = \mu(B)$, then 
	\[\widetilde{B} = \spec k[[x,y_1]]/(f_1)\]
	where $x^\mu f_1(x, y_1) = f(x, xy_1)$ by Theorem \ref{blowup_germ}(2). The exceptional divisor $E$ is cut out by $x=0$, which means that
	\[i_{\widetilde{X}_{s, q}}(\widetilde{B}, E) = \dim_k k[[x,y_1]]/(f_1, x) = \dim_k k[[y_1]]/(f_1(0, y_1)) = \mu\]
	as claimed.
	
	(3) This follows from Theorem \ref{blowup_intersection}(2) and Proposition \ref{intersection_independence} (the latter ensures that we can compute the intersection number in $\widetilde{X}_{s,q}$ rather than $\widetilde{C}_q$).
\end{proof}

We now turn to showing that in a blowup $\widetilde{X}\to X$ of an arithmetic surface at a closed point on the special fibre, the type of $X$ determines the type of $\widetilde{X}$ and vice versa. Note that in general, there are 2 types of such a blowup based on where the centre of the blowup is:
\begin{enumerate}
	\item We can have $p$ be a singular point of $(X_s)_{red}$, or
	\item We can have $p$ lie on a unique component $\Gamma$ of $X_s$, and be a smooth point of this component.
\end{enumerate}
Call these blowups of type 1 and type 2 respectively. We need to specify which kind of blowup is done to recover the type.
\begin{theorem}\label{blowup_type}
	Let $X$ be an arithmetic surface over $\mathcal{O}_K$ and $p$ a closed point of $X_s$. Denote by $f:\widetilde{X}\to X$ the blowup of $X$ at $p$. Then the type of $X$ determines the type of $\widetilde{X}$ and vice versa.
	
	More precisely: 
	\begin{itemize}
		\item Given the type $T = (\comp, d, g, \sing, \inc, \br, S, i)$ of $X$ and the centre of the blowup (as the centre point $p\in \sing$ for type 1, or as the component $\Gamma\in \comp$ containing the centre of the blowup in type 2), we can determine the type of $\widetilde{X}$.
		\item Given the type $T' = (\comp', d, g, \sing', \inc', \br', S, i)$ of $\widetilde{X}$ and the exceptional divisor $E$ of the blowup (as an element $E\in \comp'$), we can determine the type of $X$.
	\end{itemize}
\end{theorem} 
\begin{remark}
	Note in the second case, i.e. knowing the type of $\widetilde{X}$ and $E$, we still need to know that $\widetilde{X}$ arises as the blowup of an arithmetic surface, i.e. $E$ needs to be a (-1)-curve.
\end{remark}
\begin{proof}
	Suppose $X_s$ has components $\Gamma_1, \ldots, \Gamma_r$. Then the components of $\widetilde{X}_s$ are the strict transforms $\widetilde{\Gamma}_i$ of the $\Gamma_i$, along with the exceptional divisor $E$. The genera and multiplicities are related by
	\begin{align*}
		g_{\widetilde{\Gamma}} = g_\Gamma,&\qquad g_E = 0\\
		d_{\widetilde{\Gamma}} = d_\Gamma,&\qquad d_E = \sum_{p\in \Gamma}d_\Gamma\mu_p(\Gamma).
	\end{align*}
	Indeed, $\widetilde{\Gamma}$ is birational to $\Gamma$ so they have the same normalisation, and $E\cong \mathbb{P}^1$. The formulae for the multiplicities follow from \cite[Exercise 9.2.9(a)]{liu}. So the components (with multiplicity and genera) of one of $X, \widetilde{X}$ are determined in terms of the type of the other.
	
	The singular points of $(X_s)_{red}$ away from the centre $p$ and those of $(\widetilde{X}_s)_{red}$ away from the exceptional divisor $E$ are the same (the blowup is an isomorphism there), so at these points the branches, semigroups and intersection numbers are all the same. So it remains to look at $p$ and singular points on $E$.
	
	The singular points of $(\widetilde{X}_s)_{red}$ on $E$ are just where $E$ meets other components (as $E\cong \mathbb{P}^1$ is smooth), so these correspond to tangency classes of branches of $X_s$ through $p$ (in a type 2 blowup, $p$ lies on a unique component $\Gamma$ with a unique smooth branch through it). For a singular point $q\in E$, the branches through it are a unique (smooth) branch of $E$ through $q$, and the strict transforms $\widetilde{B}$ of branches $B$ in the tangency class corresponding to $q$. Their intersection numbers are given by
	\begin{align*}
		i_q(\widetilde{B}, \widetilde{B'}) &= i_p(B, B')-\mu_p(B)\mu_p(B')\\
		i_q(\widetilde{B}, E) &= \mu_p(B)
	\end{align*}
	both by Proposition \ref{blowup_arithmetic_surface}. The semigroups $S(\widetilde{B})$ and $S(B)$ are related as in Theorem \ref{blowup_semigroup}.
	
	Hence by knowing the type of $X$ we can recover the singular points, branches, intersection numbers and semigroups of $\widetilde{X}$: the singular points $q$ on $E$ just correspond to tangency classes of branches through $p$, and branches through $q$ are just the strict transforms of the branches in the tangency class (and a branch of $E$). The tangency classes can be read off of the type of $X$ by Theorem \ref{intersection_tangency}. Then the intersection numbers are as given above by Theorem \ref{blowup_arithmetic_surface}, and the semigroup $S(\widetilde{B})$ is determined by $S(B)$ as in Theorem \ref{blowup_semigroup}.
	
	Conversely knowing the type of $\widetilde{X}$ and the exceptional divisor $E$, we can recover singular points, branches, intersection numbers and semigroups of $X$. We can say that the blowup is of type 2 if and only if $E$ meets only one other component $\widetilde{\Gamma}$ at a unique point $q$, and both $E$ and $\widetilde{\Gamma}$ have a unique branch through $q$ which intersect with multiplicity 1 (in this case $p$ lies only on the component $\Gamma$ of $X_s$ and is a smooth point of it). Otherwise $p$ is a singular point of $(X_s)_{red}$, and branches through it are just images of branches through singular points $q$ on $E$ that are not branches of $E$ itself. The intersection numbers are again determined by the formulas above: note that for a branch $B$ through $p$, we know $\mu_p(B)$ as it is equal to $i_q(\widetilde{B}, E)$ which can be read off of the type of $\widetilde{X}$. This also means that we can determine $S(B)$ from $S(\widetilde{B})$, since this is possible if we also know $\mu_p(B)$ as in Remark \ref{semigroup_recovery}.
\end{proof}
In fact we can write this out explicitly as an algorithm:
\begin{algorithm}\label{algorithm_blowup}
Suppose $f:\widetilde{X}\to X$ is a blowup of arithmetic surfaces at a closed point $p\in X_s$, and that $X$ has type $T = (\comp, d, g, \sing, \inc, \br, S, i)$ and $\widetilde{X}$ has type $T' = (\comp', d, g, \sing', \inc', \br, S, i)$ (we use the same letters except for the sets of components, singular points and the incidence relation; this should cause no confusion). One can determine one in terms of the other as follows:
\begin{enumerate}
	\item Suppose we are given the type $T$ of $X$. For ease of notation put
	\begin{align*}
		\mu(B) = \min S(B)&\qquad\text{for all branches }B\in \coprod_{\substack{p, \Gamma\\p\in \Gamma}}\br_{\Gamma, p},\\
		\mu_p(\Gamma) = \sum_{B\in \br_{\Gamma,p}}\mu(B)&\qquad\text{for all }p\in \sing, \Gamma\in \comp\text{ with }p\in \Gamma,\\
		\br_p = \coprod_{p\in \Gamma} \br_{\Gamma, p}&\qquad\text{ for all }p\in \sing\\
		B\sim_p B' \text{ if } i_p(B, B')>\mu_p(B)\mu_p(B')&\qquad\text{for all }B,B'\in \br_p, p\in \sing\\
	\end{align*}
	\begin{enumerate}
		\item If the blowup is of type 1 centred at some $p\in \sing$, then $T'$ is given as follows: it has components
		\[\comp' = \{\widetilde{\Gamma}\;:\;\Gamma\in \comp\}\cup\{E\} \]
		of multiplicities
		\[d_{\widetilde{\Gamma}} = d_\Gamma,\qquad d_E = \sum_{\Gamma}d_\Gamma \mu_p(\Gamma)\]
		and genera
		\[g_{\widetilde{\Gamma}} = g_\Gamma,\qquad g_E = 0.\]
		Its singular points are
		\[\sing' = \{f^{-1}(q)\;:\;q\in \sing\setminus\{p\}\}\cup \{q_C\;:\;\text{$C$ is an equivalence class on $\br_p$ under $\sim_p$}\}\]
		and the incidence relation is given as follows: $f^{-1}(q)$ for $q\in \sing\setminus\{p\}$ lies on exactly the components $\widetilde{\Gamma}$ where $q\in \Gamma$ (and $f^{-1}(q)\not\in E$). A singular point $q_C$ corresponding to an equivalence class $C$ lies on $E$ and all components $\widetilde{\Gamma}$ such that some branch of $\Gamma$ through $p$ lies in $C$. The branches through the singular points are
		\begin{align*}
			\br_{\widetilde{\Gamma}, f^{-1}(q)} = \{\widetilde{B}\;:\; B\in \br_{\Gamma, q}\}&\text{ if }q\in \sing\setminus\{p\}\\
			\br_{\widetilde{\Gamma}, q_C} = \{\widetilde{B}\;:\;B\in \br_{\Gamma, p}\text{ and }B\in C\}&\text{ for } C \text{ an equivalence class},\\
			\br_{E, q_C} = \{E_{q_C}\}.
		\end{align*}	
		The intersection pairing and semigroups at points $f^{-1}(q)$ for $q\in \sing \setminus \{p\}$ are 
		\[i_{f^{-1}(q)}(\widetilde{B}, \widetilde{B'}) = i_q(B, B'),\qquad S(\widetilde{B}) = S(B)\]
		(the same as in $T$), while at points $q_C$ corresponding to equivalence classes they are given as
		\begin{align*}
			i_{q_C}(\widetilde{B}, \widetilde{B'}) &= i_p(B, B') - \mu_p(B)\mu_p(B')\\
			i_{q_C}(\widetilde{B}, E_{q_C}) &= \mu_p(B),
		\end{align*}
		moreover $S(E_q) = \mathbb{Z}_{>0}$ and if $S(B)$ has $\mu = \mu_p(S)$-Apéry set $0\le a_0<a_1<\ldots<a_{\mu-1}$ then $S(\widetilde{B})$ has $\mu$-Apéry set $0\le a_0<a_1 + \mu <\ldots < a_{\mu-1}+(\mu-1)\mu$.
		\item If the blowup is of type 2 centred at a point lying on a unique component $\Gamma_p$, then $T'$ is given as follows: it has components
		\[\comp' = \{\widetilde{\Gamma}\;:\;\Gamma\in \comp\}\cup \{E\}\]
		of multiplicities
		\[d_{\widetilde{\Gamma}} = d_\Gamma,\qquad d_E = d_{\Gamma_p}\]
		and genera
		\[g_{\widetilde{\Gamma}} = g_\Gamma,\qquad g_E = 0.\]
		Its singular points are
		\[\sing' = \{f^{-1}(q)\;:\;q\in \sing\} \cup \{p\}\]
		and the incidence relation is given as follows: $f^{-1}(q)$ for $q\in \sing$ lies exactly on components $\widetilde{\Gamma}$ with $q\in \Gamma$ (and doesn't lie on $E$). The new singular point $p$ lies exactly on $E$ and $\widetilde{\Gamma_p}$. The branches through the singular points are
		\begin{align*}
			\br_{\widetilde{\Gamma}, f^{-1}(q)} &= \{\widetilde{B}\;:\;B\in \br_{\Gamma, q}\}\qquad\text{for }q\in \sing,\\
			\br_{\widetilde{\Gamma_p}, p} &= \{B_p\},\\
			\br_{E, p} &= \{E_p\}.
		\end{align*}
		The intersection pairing and semigroups at points $f^{-1}(q)$ for $q\in \sing$ are 
		\[i_{f^{-1}(q)}(\widetilde{B}, \widetilde{B'}) = i_q(B, B'),\qquad S(\widetilde{B}) = S(B)\]
		(the same as in $T$), and at $p$ they are given as
		\begin{align*}
			i_p(B_p, E_p) = 1,\\
			S(B_p) = S(E_p) = \mathbb{Z}_{>0}.
		\end{align*}
	\end{enumerate}
	\item Now suppose we are given the type $T'$ of $\widetilde{X}$ and an exceptional divisor $E\in \comp'$. Then we say $E$ is a \emph{tail} if it meets only one other component $\Gamma\in \comp'$ at a unique singular point $p$, and both $E$ and $\Gamma$ have a unique smooth branch through $p$ which intersect transversely i.e. with intersection number 1.
	\begin{enumerate}
		\item If $E$ is not a tail, the blowup is of type 1, and $T$ is given as follows: its components are
		\[\comp = \{f(\Gamma)\;:\;\Gamma\in \comp'\setminus\{E\}\}\]
		with multiplicities and genera
		\[d_{f(\Gamma)} = d_\Gamma,\qquad g_{f(\Gamma)} = g_\Gamma.\]
		Its singular points are
		\[\sing = \{f(q)\;:\;q\in \sing', q\not\in E\}\cup \{p\}\]
		and the incidence relation is given as follows: for $q\in \sing', q\not\in E$ we have $f(q)\in f(\Gamma)$ if and only if $q\in \Gamma$, and $p\in f(\Gamma)$ if and only if $\Gamma$ intersects $E$ i.e. $E$ and $\Gamma$ have a singular point $q\in \sing'$ in common (so we think $p = f(E)$). The branches through the singular points are
		\begin{align*}
			\br_{f(\Gamma), f(q)} &= \{f(B)\;:\;B\in \br_{\Gamma, q}\}\qquad\text{for }q\in \sing',q\not\in E,\\
			\br_{f(\Gamma), p} &= \{f(B)\;:\; B\in \coprod_{q\in \sing', q\in E}\br_{\Gamma, q}\}.
		\end{align*}
		The intersection pairing and semigroups at $f(q)$ for $q\in \sing', q\not\in E$ are given as
		\[i_{f(q)}(f(B), f(B')) = i_q(B, B'),\qquad S(f(B)) = S(B)\]
		(the same as in $T'$), and at $p$ the intersection pairing is given by
		\[i_p(f(B), f(B')) = \begin{cases}
			i_q(B, B') + i_q(B, E) + i_q(B', E)&\text{ if }B,B'\text{ are both branches at the same }q\in E\\
			i_q(B, E) + i_{q'}(B', E)&\text{ if }B,B'\text{ are branches at distinct points }q,q'\in E,\\
		\end{cases}\]
		and the semigroups are given as follows: if $B\in \br_{\Gamma, q}$ for some $q\in \sing', q\in E$, we put $\mu = i_q(B, E)$. Then if $S(B)$ has $\mu$-Apéry set $0\le a_0<a_1<\ldots<a_{\mu-1}$, then $S(f(B))$ has $\mu$-Apéry set $0\le a_0<a_1-\mu<\ldots<a_{\mu-1} - \mu(\mu-1)$.
		\item If $E$ is a tail, then the blowup is of type 2, and $T$ is given as follows: its components are
		\[\comp = \{f(\Gamma)\;:\;\Gamma\in \comp'\setminus\{E\}\}\]
		with multiplicities and genera
		\[d_{f(\Gamma)} = d_\Gamma,\qquad g_{f(\Gamma)} = g_\Gamma.\]
		Its singular points are
		\[\sing = \{f(q)\;:\;q\in \sing', q\not\in E\}\]
		and the incidence relation is given as follows: for $q\in \sing', q\not\in E$ we have $f(q)\in f(\Gamma)$ if and only if $q\in \Gamma$. The branches through the singular points are
		\[\br_{f(\Gamma), f(q)} = \{f(B)\;:\;B\in \br_{\Gamma, q}\}\qquad\text{for }q\in \sing',q\not\in E.\]
		The intersection pairing and semigroups are given as
		\[i_{f(q)}(f(B), f(B')) = i_q(B, B'),\qquad S(f(B)) = S(B)\]
		(the same as in $T'$).
	\end{enumerate}
\end{enumerate}
\end{algorithm}

\subsection{Reduction types}
Now we turn to defining the reduction type of a (smooth, projective) curve over $K$.

Let $C$ be a smooth projective curve over $K$ of genus $g\ge 1$. Then regular models $X$ of $C$ are arithmetic surfaces over $\mathcal{O}_K$. There are some canonical choices of models of $C$, of which we treat the following two (see \cite[Proposition 10.1.8]{liu}):
\begin{itemize}
	\item There is a unique \emph{minimal regular model} $X_{\mathrm{min}}$ of $C$, with the universal property that for any regular model $X$ of $C$, there is a unique morphism of models $X\to X_{\mathrm{min}}$.
	\item There is also a unique \emph{minimal regular model with normal crossings} $X_{\mathrm{rnc}}$ of $C$, with the universal property that its special fibre is a normal crossings divisor, and for any other regular model $X$ such that $X_s$ is a normal crossings divisor, there is a unique morphism of models $X\to X_{\mathrm{rnc}}$.
\end{itemize}
One can define the reduction type of the curve in terms of the type of any one of these models:
\begin{definition}
	Let $C$ be a smooth projective curve over $K$ of genus $g\ge 1$. The \emph{minimal regular reduction type} of $C$ is the type of the minimal regular model $X_{\mathrm{min}}$ of $C$. The \emph{minimal normal crossings reduction type} of $C$ is the type of the minimal regular normal crossings model $X_{\mathrm{rnc}}$ of $C$.
\end{definition}
Note that the minimal normal crossings type of $C$ is much more constrained: any singular point $p$ on the reduced special fibre of $X_{\mathrm{rnc}}$ lies on either one component (a node) or on two components, and in both cases has two smooth branches through it which intersect transversely. This therefore determines the branches, their intersection numbers and semigroups, and we just need to keep track of the configuration of components on the special fibre, along with multiplicities and genera, and also how many nodes are on each component.

Our main result is the following:
\begin{theorem}\label{min_reg<->rnc}
	Let $C$ be a smooth projective curve over $K$ of genus $g\ge 1$. Then the minimal regular reduction type of $C$ determines the minimal normal crossings reduction type and vice versa. Both of them are determined by the type of any regular model $X$ of $C$.
	There is an explicit algorithm to determine the type of any regular model from the minimal regular type or vice versa. 
\end{theorem}
\begin{proof}
	Let $X_{\mathrm{min}}$ be the minimal regular model of $C$ and $X$ be an arbitrary regular model of $C$. By the universal property of $X_{\mathrm{min}}$ there is a morphism of models $X\to X_{\mathrm{min}}$. By \cite[Theorem 9.2.2]{liu} this morphism can be written as a composition of finitely many blowups along closed points. Moreover the map is an isomorphism on the generic fibre, so each blowup is centred on the special fibre. Therefore there is a sequence 
	\[X=X_n\to X_{n-1}\to \ldots \to X_1 \to X_0=X_{\mathrm{min}}\tag{$\dagger$}\]
	where each $X_{i+1}\to X_i$ is a blowup at a closed point on the special fibre. So by successive applications of Theorem \ref{blowup_type}, we see that the type of $X_{\mathrm{min}}$ determines the type of $X$ and vice versa. Putting $X = X_{\mathrm{rnc}}$ we get that the minimal regular type determines the minimal normal crossings type and vice versa.
	
	Moreover the sequence $(\dagger)$ can be found by contracting $(-1)$-curves on the special fibre of $X$ until none remain (\cite[Proposition 3.19]{liu}). Then applying Algorithm \ref{algorithm_blowup} repeatedly to these blowups gives an algorithm determining the type of $X$ in terms of the type of $X_{\mathrm{min}}$ or vice versa.
\end{proof}
This means that we can define the \emph{reduction type} as being either of these reduction types:
\begin{definition}
	The \emph{reduction type} of a smooth projective curve $C/K$ is its minimal regular reduction type (or its minimal normal crossings reduction type).
\end{definition}
It also follows that the reduction type of a curve can be read off of the special fibre of any regular model.
\section{Classification in low genus}\label{section_low_genus}
Let $\mathcal{O}_K$ be a DVR with fraction field $K$ and residue field $k$. We assume $k$ is algebraically closed. In this section we will look at minimal regular/minimal regular normal crossings models of smooth projective curves $C/K$ of low genus, and classify them by their reduction type.

We do this by using the classification of Dokchitser \cite{tim_classification} in terms of the minimal normal crossings model, and using our earlier result to translate this into a classification in terms of the minimal regular model. First we need to show that the notion of a reduction type in loc.cit. yields a type in our sense:
\begin{definition}[$\approx$ Definition 3.6 of a reduction type in \cite{tim_classification}]
	A \emph{normal crossings reduction type} is a tuple $R = (\comp,d,p,g,\cdot)$ with
	\begin{itemize}
		\item $\comp$ is a nonempty finite set whose elements $\Gamma$ are called components,
		\item $d:\comp\to \mathbb{Z}_{\ge 0}$ and $p,g:\comp\to \mathbb{Z}_{\ge 0}$ are functions sending $\Gamma$ to its multiplicity $d_\Gamma$ and arithmetic/geometric genera $p_\Gamma, g_\Gamma$ with $g_\gamma\le p_\Gamma$,
		\item $\cdot: \mathbb{Z}\comp\times \mathbb{Z}\comp\to \mathbb{Z}$ is a symmetric bilinear pairing with $\Gamma\cdot \Gamma'\ge 0$ for $\Gamma\not=\Gamma'$, and such that the left kernel of $\cdot$ is free of rank one and contains $\sum d_\Gamma \Gamma$.
	\end{itemize}
\end{definition}
Note we don't require reduction types to satisfy a minimality condition, this will not change any result in what follows.

Given an arithmetic surface $X$ such that $X_s$ is a normal crossings divisor, it gives rise to a normal crossings reduction type in the obvious way, taking $\comp$ to be its set of components, $d,p,g$ to be their multiplicities and arithmetic/geometric genera, and $\cdot$ the intersection pairing on components. From this notion we can extract a type in our sense:

\begin{definition}
	Let $R = (\comp, d, p, g, \cdot)$ be a normal crossings reduction type. The \emph{associated type} is given by $T = (\comp, d, g, \sing, \inc, \br, S, i)$ where
	\begin{itemize}
		\item $\comp$ and $d,g$ are as in $R$,
		\item the set of singular points $\sing$ is the disjoint union of sets $N_\Gamma$ and $I_{\{\Gamma, \Gamma'\}}$ where
			\subitem for all $\Gamma\in S$, a set $N_\Gamma$ of nodes of $\Gamma$, where $|N_\Gamma| = p_\Gamma - g_\Gamma$,
			\subitem for all unordered pairs $\{\Gamma,\Gamma'\}$ of distinct elements of $S$, a set $I_{\{\Gamma, \Gamma'\}}$ of intersection points of size $|I_{\{\Gamma, \Gamma'\}}| = \Gamma\cdot \Gamma'$.
		The incidence relation is given as follows: any $p\in N_\Gamma$ lies only on $\Gamma$, and any $p\in I_{\{\Gamma, \Gamma'\}}$ lies only on $\Gamma$ and $\Gamma'$.
		\item We have the branches
		\begin{align*}
			\br_{\Gamma, p} = \{\Gamma_{p, 1}, \Gamma_{p,1}\}\text{ for }p\in N_\Gamma\\
			\br_{\Gamma, p} = \{\Gamma_p\}\text{ for }p\in I_{\{\Gamma, \Gamma'\}}.
		\end{align*}
		Every intersection number between distinct branches is 1 (i.e. every intersection is transversal) and the semigroup of every branch is $\mathbb{Z}_{>0}$ (i.e. every branch is smooth).
	\end{itemize}
\end{definition}
It's clear that given an arithmetic surface $X$, the associated type of its normal crossings reduction type is exactly the type of $X$ (i.e. the above association is compatible with taking types of arithmetic surfaces).
\subsection{Genus 1}
Let $X$ be the minimal regular model of a smooth projective genus 1 curve over $K$. Then the numerical type of $X$ satisfies
\[\sum_\Gamma d_\Gamma(2p_\Gamma-2-\Gamma\cdot\Gamma) = 0\]
by the adjunction formula. This means that we can take `multiples' of types in genus 1:
\begin{definition}
	Let $T = (\comp, d, g, \sing, \inc, \br, S, i)$ be a type and $m\ge 1$ an integer. Define the type $[m]T$ as 
	\[[m]T = (\comp, d', g, \sing, \inc, \br, S, i)\]
	where $d'_\Gamma = md_\Gamma$ for all $\Gamma\in \comp$. Hence we just multiply the multiplicity of every component by $m$.
\end{definition}
It is a classical result of \cite{Kodaira} and Néron \cite{Neron} that (in our terminology) the minimal regular reduction type of an elliptic curve (or genus 1 curve with a rational point) is given by one of the \emph{Kodaira types}, which are as follows:

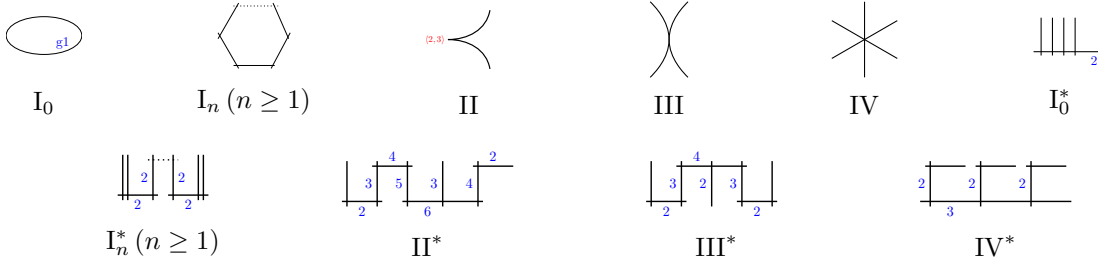
\begin{figure}[h!]
	\centering
	\tikz{\node[label=-90:{$\mathrm{I}_0$}](1){\resizebox{!}{1cm}{
				\begin{tikzpicture}
					\clip (-1, -1) rectangle + (2,2);
					\draw (0,0) ellipse (1 and 0.5);
					\node[blue] at (0.5, -0.2) {g1};
				\end{tikzpicture}
			}
		};
	}
	\hspace*{1cm}
	\tikz{\node[label=-90:{$\mathrm{I}_n\,(n\ge 1)$}](1){\resizebox{!}{1cm}{
				\begin{tikzpicture}
					\clip (-0.6,-0.1) rectangle + (2.2, 2.2);
					\draw[thick] (-0.1, 0) to (1.1, 0);
					\draw[thick, dotted] (-0.1, 1.732) to (1.2, 1.732);
					\draw[thick] (0.05, -0.086) to (-0.55, 0.952);
					\draw[thick] (0.995, -0.086) to (1.55, 0.952);
					\draw[thick] (-0.55, 0.779) to (0.05, 1.818);
					\draw[thick] (1.505, 0.779) to (1.05, 1.818);
				\end{tikzpicture}
			}
		};
	}
	\hspace*{1cm}
	\tikz{\node[label=-90:{$\mathrm{II}$}](1){\resizebox{!}{1cm}{
				\begin{tikzpicture}
					\clip (-0.5, -0.9) rectangle + (2,1.8);
					\node[left, red, scale = 0.6] at (0,0){$\left\langle 2,3\right\rangle$};
					\draw[thick] (1,0.7) to [out = -100, in = 0] (0,0) to [out = 0, in = 100] (1, -0.7);
				\end{tikzpicture}
			}
		};
	}
	\hspace*{1cm}
	\tikz{\node[label=-90:$\mathrm{III}$](1){\resizebox{!}{1cm}{
				\begin{tikzpicture}
					\clip (-1,-1) rectangle + (2, 2);
					\draw[rotate = -90, thick] (-1, 0.5) parabola bend(0,0) (1, 0.5);
					\draw[rotate = 90, thick] (-1, 0.5) parabola bend(0,0) (1, 0.5);
				\end{tikzpicture}
			}
		};
	}
	\hspace*{1cm}
	\tikz{\node[label=-90:$\mathrm{IV}$](1){\resizebox{!}{1cm}{
				\begin{tikzpicture}
					\clip (-1,-1) rectangle + (2, 2);
					\draw[thick] (0,1) to (0, -1);
					\draw[thick] (-0.866, 0.5) to (0.866, -0.5);
					\draw[thick] (-0.866, -0.5) to (0.866, 0.5);
				\end{tikzpicture}
			}
		};
	}
	\hspace*{1cm}
	\tikz{\node[label=-90:$\mathrm{I}_0^*$](1){\resizebox{!}{1cm}{
				\begin{tikzpicture}
					\clip (-0.3,-0.5) rectangle + (2, 2);
					\draw[thick] (0,0) to (1.8, 0);
					\draw[thick] (0.2, -0.1) to (0.2, 0.9);
					\draw[thick] (0.5, -0.1) to (0.5, 0.9);
					\draw[thick] (0.8, -0.1) to (0.8, 0.9);
					\draw[thick] (1.1, -0.1) to (1.1, 0.9);
					\node[blue, below] at (1.6, 0) {2};
				\end{tikzpicture}
			}
		};
	}
	\hspace*{1cm}
	\tikz{\node[label=-90:$\mathrm{I}_n^*\,(n\ge 1)$](1){\resizebox{!}{1cm}{
				\begin{tikzpicture}
					\clip (-0.1, -0.3) rectangle + (2,1.5);
					\draw[thick](0,0) to (0.8, 0);
					\draw[thick](0.1, -0.1) to (0.1, 0.8);
					\draw[thick](0.2, -0.1) to (0.2, 0.8);
					\draw[thick](0.7, -0.1) to (0.7, 0.8);
					\draw[thick, dotted](0.6, 0.7) to (1.2, 0.7);
					\draw[thick](1.1, 0.8) to (1.1, -0.1);
					\draw[thick](1, 0) to (1.8, 0);
					\draw[thick](1.6, -0.1) to (1.6, 0.8);
					\draw[thick](1.7, -0.1) to (1.7, 0.8);
					\node[blue, below, scale = 0.8] at (0.4, 0) {2};
					\node[blue, left, scale = 0.8] at (0.7, 0.35) {2};
					\node[blue, right, scale = 0.8] at (1.1, 0.35) {2};
					\node[blue, below, scale = 0.8] at (1.4, 0) {2};
				\end{tikzpicture}
			}
		};
	}
	\hspace*{1cm}
	\tikz{\node[label=-90:$\mathrm{II}^*$](1){\resizebox{!}{1cm}{
				\begin{tikzpicture}
					\clip (-0.1, -0.3) rectangle + (3.6, 1.5);
					\draw[thick] (0,0) to (0.8, 0);
					\draw[thick] (0.1, -0.1) to (0.1, 0.8);
					\draw[thick] (0.7, -0.1) to (0.7, 0.8);
					\draw[thick] (0.6, 0.7) to (1.4, 0.7);
					\draw[thick] (1.3, 0.8) to (1.3, -0.1);
					\draw[thick] (1.2, 0) to (2.8, 0);
					\draw[thick] (2, -0.1) to (2, 0.8);
					\draw[thick] (2.7, -0.1) to (2.7, 0.8);
					\draw[thick] (2.6, 0.7) to (3.4, 0.7);
					\node[blue, below, scale = 0.8] at (0.4, 0) {2};
					\node[blue, left, scale = 0.8] at (0.7, 0.35) {3};
					\node[blue, above, scale = 0.8] at (1, 0.7) {4};
					\node[blue, left, scale = 0.8] at (1.3, 0.35) {5};
					\node[blue, below, scale = 0.8] at (1.7, 0) {6};
					\node[blue, left, scale = 0.8] at (2, 0.35) {3};
					\node[blue, left, scale = 0.8] at (2.7, 0.35) {4};
					\node[blue, above, scale = 0.8] at (3, 0.7) {2};
				\end{tikzpicture}
			}
		};
	}
	\hspace*{1cm}
	\tikz{\node[label=-90:$\mathrm{III}^*$](1){\resizebox{!}{1cm}{
				\begin{tikzpicture}
					\clip (-0.1, -0.3) rectangle + (3, 1.5);
					\draw[thick] (0,0) to (0.8, 0);
					\draw[thick] (0.1, -0.1) to (0.1, 0.8);
					\draw[thick] (0.7, -0.1) to (0.7, 0.8);
					\draw[thick] (0.6, 0.7) to (2, 0.7);
					\draw[thick] (1.3, -0.1) to (1.3, 0.8);
					\draw[thick] (1.9, -0.1) to (1.9, 0.8);
					\draw[thick] (1.8, 0) to (2.6, 0);
					\draw[thick] (2.5, -0.1) to (2.5, 0.8);
					\node[blue, below, scale = 0.8] at (0.4, 0) {2};
					\node[blue, left, scale = 0.8] at (0.7, 0.35) {3};
					\node[blue, above, scale = 0.8] at (1, 0.7) {4};
					\node[blue, left, scale = 0.8] at (1.3, 0.35) {2};
					\node[blue, left, scale = 0.8] at (1.9, 0.35) {3};
					\node[blue, below, scale = 0.8] at (2.2, 0) {2};
				\end{tikzpicture}
			}
		};
	}
		\hspace*{1cm}
	\tikz{\node[label=-90:$\mathrm{IV}^*$](1){\resizebox{!}{1cm}{
				\begin{tikzpicture}
					\clip (-0.1, -0.3) rectangle + (3.2, 1.5);
					\draw[thick] (0,0) to (3, 0);
					\draw[thick] (0.2, -0.1) to (0.2, 0.8);
					\draw[thick] (0.1, 0.7) to (0.9, 0.7);
					\draw[thick] (1.2, -0.1) to (1.2, 0.8);
					\draw[thick] (1.1, 0.7) to (1.9, 0.7);
					\draw[thick] (2.2, -0.1) to (2.2, 0.8);
					\draw[thick] (2.1, 0.7) to (2.9, 0.7);
					\node[blue, below, scale = 0.8] at (0.6, 0) {3};
					\node[blue, left, scale = 0.8] at (0.2, 0.35) {2};
					\node[blue, left, scale = 0.8] at (1.2, 0.35) {2};
					\node[blue, left, scale = 0.8] at (2.2, 0.35) {2};
				\end{tikzpicture}
			}
		};
	}
\end{figure}
We have moreover the minimal regular reduction type of an arbitrary genus 1 curve is some multiple $[m]\mathrm{Kod}$ where $\mathrm{Kod}$ is a Kodaira type \cite{saito1987vanishing}.

We recover this classification using our notion of reduction type, and we can also determine the normal crossings reduction type in each case.
\begin{theorem}
	Let $C/K$ be a smooth projective curve of genus 1. Then $C$ has reduction type $[m]\mathrm{Kod}$ for some $m\ge 1$ and $\mathrm{Kod}$ a Kodaira type. The minimal regular and minimal normal crossings reduction types are given as follows:
	\begin{table}[H]
	\centering
	\begin{tabular}{c|c|c}
		Kodaira type   & Minimal regular & Minimal normal crossings \\
		\hline
		$\mathrm{I}_0$ & \multicolumn{2}{c}{\begin{tikzpicture}[baseline = (A.base), scale = 0.7]
			\clip (-1, -1) rectangle + (2,2);
			\draw (0,0) ellipse (1 and 0.5);
			\node[blue] at (0.5, -0.2) {g1};
			\node at (0,0)(A){};
		\end{tikzpicture}} \\            
		\hline
		$\mathrm{I}_n$ & \multicolumn{2}{c}{\begin{tikzpicture}[baseline = (A.base), scale = 0.7]
				\clip (-0.6,-0.3) rectangle + (2.2, 2.4);
				\node at (0.966, 1)(A){};
				\draw[thick] (-0.1, 0) to (1.1, 0);
				\draw[thick, dotted] (-0.1, 1.732) to (1.2, 1.732);
				\draw[thick] (0.05, -0.086) to (-0.55, 0.952);
				\draw[thick] (0.995, -0.086) to (1.55, 0.952);
				\draw[thick] (-0.55, 0.779) to (0.05, 1.818);
				\draw[thick] (1.505, 0.779) to (1.05, 1.818);
		\end{tikzpicture}}\\
		\hline
		$\mathrm{II}$ & \begin{tikzpicture}[baseline = (A.base), scale = 0.7]
			\clip (-0.5, -0.9) rectangle + (2,1.8);
			\node at (0,0)(A){};
			\node[red, left, scale = 0.4] at (0,0){$\left\langle 2,3\right\rangle$};
			\draw[thick] (1,0.7) to [out = -100, in = 0] (0,0) to [out = 0, in = 100] (1, -0.7);
		\end{tikzpicture} & \begin{tikzpicture}[baseline = (B.base), scale = 0.7]
		\clip (-0.1, -0.3) rectangle + (2.2, 1.7);
		\node at (1, 0.6)(B){};
		\draw[thick] (0,0) to (2, 0);
		\draw[thick] (0.2, -0.1) to (0.2, 1);
		\draw[thick] (0.6, -0.1) to (0.6, 1);
		\draw[thick] (1, -0.1) to (1, 1);
		\node[above, scale = 0.6, blue] at (1,1){3};
		\node[above, scale = 0.6, blue] at (0.6,1){2};
		\node[below, scale = 0.6, blue] at (1.5,0) {6};
		\end{tikzpicture}\\
		\hline
		$\mathrm{III}$& \begin{tikzpicture}[baseline = (A.base), scale = 0.7]
			\clip (-1,-1.2) rectangle + (2, 2.4);
			\node at (0,0)(A){};
			\draw[rotate = -90, thick] (-1, 0.5) parabola bend(0,0) (1, 0.5);
			\draw[rotate = 90, thick] (-1, 0.5) parabola bend(0,0) (1, 0.5);
		\end{tikzpicture}& \begin{tikzpicture}[baseline = (B.base), scale = 0.7]
		\clip (-0.1, -0.3) rectangle + (2.2, 1.7);
		\node at (1, 0.6)(B){};
		\draw[thick] (0,0) to (2, 0);
		\draw[thick] (0.2, -0.1) to (0.2, 1);
		\draw[thick] (0.6, -0.1) to (0.6, 1);
		\draw[thick] (1, -0.1) to (1, 1);
		\node[above, scale = 0.6, blue] at (1,1){2};
		\node[below, scale = 0.6, blue] at (1.5,0) {4};
		\end{tikzpicture}\\
		\hline
		$\mathrm{IV}$&\begin{tikzpicture}[baseline = (A.base), scale = 0.7]
			\clip (-1,-1.2) rectangle + (2, 2.4);
			\node at (0,0)(A){};
			\draw[thick] (0,1) to (0, -1);
			\draw[thick] (-0.866, 0.5) to (0.866, -0.5);
			\draw[thick] (-0.866, -0.5) to (0.866, 0.5);
		\end{tikzpicture}&\begin{tikzpicture}[baseline = (B.base), scale = 0.7]
		\clip (-0.1, -0.3) rectangle + (2.2, 1.7);
		\node at (1, 0.6)(B){};
		\draw[thick] (0,0) to (2, 0);
		\draw[thick] (0.2, -0.1) to (0.2, 1);
		\draw[thick] (0.6, -0.1) to (0.6, 1);
		\draw[thick] (1, -0.1) to (1, 1);
		\node[below, scale = 0.6, blue] at (1.5,0) {3};
		\end{tikzpicture}\\
		\hline
		$\mathrm{I}_0^*$&\multicolumn{2}{c}{\begin{tikzpicture}[baseline = (A.base), scale = 0.7]
				\clip (-0.3,-0.5) rectangle + (2, 2);
				\node at (0.9, 0.35)(A){};
				\draw[thick] (0,0) to (1.8, 0);
				\draw[thick] (0.2, -0.1) to (0.2, 0.9);
				\draw[thick] (0.5, -0.1) to (0.5, 0.9);
				\draw[thick] (0.8, -0.1) to (0.8, 0.9);
				\draw[thick] (1.1, -0.1) to (1.1, 0.9);
				\node[blue, below, scale = 0.8] at (1.6, 0) {2};
		\end{tikzpicture}}\\
	\hline
	$\mathrm{I}_n^*$&\multicolumn{2}{c}{\begin{tikzpicture}[baseline = (A.base), scale = 0.7]
			\clip (-0.1, -0.5) rectangle + (2,1.9);
			\node at (0.9, 0.35)(A){};
			\draw[thick](0,0) to (0.8, 0);
			\draw[thick](0.1, -0.1) to (0.1, 0.8);
			\draw[thick](0.2, -0.1) to (0.2, 0.8);
			\draw[thick](0.7, -0.1) to (0.7, 0.8);
			\draw[thick, dotted](0.6, 0.7) to (1.2, 0.7);
			\draw[thick](1.1, 0.8) to (1.1, -0.1);
			\draw[thick](1, 0) to (1.8, 0);
			\draw[thick](1.6, -0.1) to (1.6, 0.8);
			\draw[thick](1.7, -0.1) to (1.7, 0.8);
			\node[blue, below, scale = 0.8] at (0.4, 0) {2};
			\node[blue, left, scale = 0.8] at (0.7, 0.35) {2};
			\node[blue, right, scale = 0.8] at (1.1, 0.35) {2};
			\node[blue, below, scale = 0.8] at (1.4, 0) {2};
	\end{tikzpicture}}\\
	\hline
	$\mathrm{II}^*$&\multicolumn{2}{c}{\begin{tikzpicture}[baseline = (A.base), scale = 0.7]
			\clip (-0.1, -0.5) rectangle + (3.6, 1.9);
			\node at (0, 0.35)(A){};
			\draw[thick] (0,0) to (0.8, 0);
			\draw[thick] (0.1, -0.1) to (0.1, 0.8);
			\draw[thick] (0.7, -0.1) to (0.7, 0.8);
			\draw[thick] (0.6, 0.7) to (1.4, 0.7);
			\draw[thick] (1.3, 0.8) to (1.3, -0.1);
			\draw[thick] (1.2, 0) to (2.8, 0);
			\draw[thick] (2, -0.1) to (2, 0.8);
			\draw[thick] (2.7, -0.1) to (2.7, 0.8);
			\draw[thick] (2.6, 0.7) to (3.4, 0.7);
			\node[blue, below, scale = 0.8] at (0.4, 0) {2};
			\node[blue, left, scale = 0.8] at (0.7, 0.35) {3};
			\node[blue, above, scale = 0.8] at (1, 0.7) {4};
			\node[blue, left, scale = 0.8] at (1.3, 0.35) {5};
			\node[blue, below, scale = 0.8] at (1.7, 0) {6};
			\node[blue, left, scale = 0.8] at (2, 0.35) {3};
			\node[blue, left, scale = 0.8] at (2.7, 0.35) {4};
			\node[blue, above, scale = 0.8] at (3, 0.7) {2};
	\end{tikzpicture}}\\
	\hline
	$\mathrm{III}^*$&\multicolumn{2}{c}{\begin{tikzpicture}[baseline = (A.base), scale = 0.7]
			\clip (-0.1, -0.5) rectangle + (3, 1.9);
			\node at (0, 0.35)(A){};
			\draw[thick] (0,0) to (0.8, 0);
			\draw[thick] (0.1, -0.1) to (0.1, 0.8);
			\draw[thick] (0.7, -0.1) to (0.7, 0.8);
			\draw[thick] (0.6, 0.7) to (2, 0.7);
			\draw[thick] (1.3, -0.1) to (1.3, 0.8);
			\draw[thick] (1.9, -0.1) to (1.9, 0.8);
			\draw[thick] (1.8, 0) to (2.6, 0);
			\draw[thick] (2.5, -0.1) to (2.5, 0.8);
			\node[blue, below, scale = 0.8] at (0.4, 0) {2};
			\node[blue, left, scale = 0.8] at (0.7, 0.35) {3};
			\node[blue, above, scale = 0.8] at (1, 0.7) {4};
			\node[blue, left, scale = 0.8] at (1.3, 0.35) {2};
			\node[blue, left, scale = 0.8] at (1.9, 0.35) {3};
			\node[blue, below, scale = 0.8] at (2.2, 0) {2};
	\end{tikzpicture}}\\
	\hline
	$\mathrm{IV}^*$&\multicolumn{2}{c}{\begin{tikzpicture}[baseline = (A.base), scale = 0.7]
			\clip (-0.3, -0.5) rectangle + (3.4, 1.9);
			\node at (0, 0.35)(A){};
			\draw[thick] (0,0) to (3, 0);
			\draw[thick] (0.2, -0.1) to (0.2, 0.8);
			\draw[thick] (0.1, 0.7) to (0.9, 0.7);
			\draw[thick] (1.2, -0.1) to (1.2, 0.8);
			\draw[thick] (1.1, 0.7) to (1.9, 0.7);
			\draw[thick] (2.2, -0.1) to (2.2, 0.8);
			\draw[thick] (2.1, 0.7) to (2.9, 0.7);
			\node[blue, below, scale = 0.8] at (0.6, 0) {3};
			\node[blue, left, scale = 0.8] at (0.2, 0.35) {2};
			\node[blue, left, scale = 0.8] at (1.2, 0.35) {2};
			\node[blue, left, scale = 0.8] at (2.2, 0.35) {2};
	\end{tikzpicture}}\\
	\hline
	\end{tabular}
	\end{table}
	
(This means that for a curve of type $[m]\mathrm{Kod}$, its minimal regular and minimal normal crossings types are given by $m$ times the corresponding entry).
\end{theorem}\label{genus_1_classification}
\begin{proof}
	The classification of minimal normal crossings reduction types is by \cite[Corollary 9.1(1,2)]{tim_classification}. By Theorem \ref{min_reg<->rnc} this also determines the minimal regular type, and we can determine these by Algorithm \ref{algorithm_blowup}. The only minimal normal crossings reduction types containing (-1)-curves are types $\mathrm{II},\mathrm{III}$ and $\mathrm{IV}$, which change as follows under successive blowing down of these components:
	
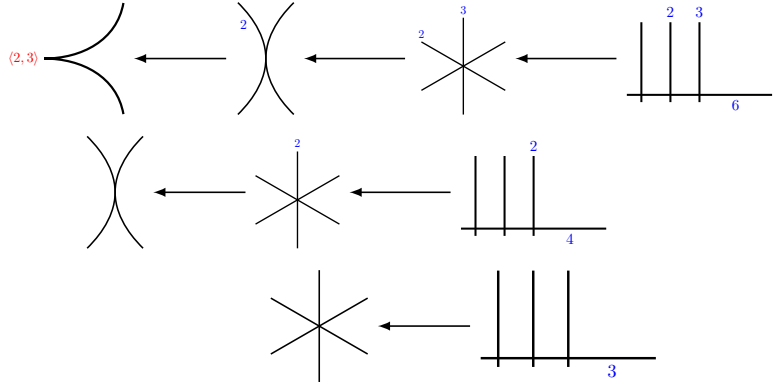
\begin{figure}[H]
		\centering
		\tikz[remember picture]{\node(1){%
				\resizebox{!}{1.5cm}{%
					\begin{tikzpicture}
						\draw[thick] (1,0.7) to [out = -100, in = 0] (0,0) to [out = 0, in = 100] (1, -0.7);
						\node[left, scale = 0.5, red] at (0,0){$\left\langle2,3\right\rangle$};
			\end{tikzpicture}}};}
		\hspace*{1cm}
		\tikz[remember picture]{\node(2){%
					\resizebox{!}{1.5cm}{%
						\begin{tikzpicture}
							\draw[rotate = -90, thick] (-1, 0.5) parabola bend(0,0) (1, 0.5);
							\draw[rotate = 90, thick] (-1, 0.5) parabola bend(0,0) (1, 0.5);
							\node[left, scale = 0.7, blue] at (-0.25, 0.6) {2};
				\end{tikzpicture}}};}
			\tikz[overlay, remember picture]{
				\draw[latex-, semithick] (1) -- (1-|2.west);}
		\hspace*{1cm}
		\tikz[remember picture]{\node(3){%
					\resizebox{!}{1.5cm}{%
						\begin{tikzpicture}
							\draw[thick] (0,1) to (0, -1);
							\draw[thick] (-0.866, 0.5) to (0.866, -0.5);
							\draw[thick] (-0.866, -0.5) to (0.866, 0.5);
							\node[above, scale = 0.7, blue] at (0,1){3};
							\node[above, scale = 0.7, blue] at (-0.866, 0.5){2};
				\end{tikzpicture}}};}
			\tikz[overlay, remember picture]{
				\draw[latex-, semithick] (2) -- (2-|3.west);}
		\hspace*{1cm}
		\tikz[remember picture]{\node(4){%
					\resizebox{!}{1.5cm}{%
						\begin{tikzpicture}
							\draw[thick] (0,0) to (2, 0);
							\draw[thick] (0.2, -0.1) to (0.2, 1);
							\draw[thick] (0.6, -0.1) to (0.6, 1);
							\draw[thick] (1, -0.1) to (1, 1);
							\node[above, scale = 0.6, blue] at (1,1){3};
							\node[above, scale = 0.6, blue] at (0.6,1){2};
							\node[below, scale = 0.6, blue] at (1.5,0) {6};
				\end{tikzpicture}}};}
			\tikz[overlay, remember picture]{
				\draw[latex-, semithick] (3) -- (3-|4.west);
		}
	\newline
	\tikz[remember picture]{\node(2){%
			\resizebox{!}{1.5cm}{%
				\begin{tikzpicture}
					\draw[rotate = -90, thick] (-1, 0.5) parabola bend(0,0) (1, 0.5);
					\draw[rotate = 90, thick] (-1, 0.5) parabola bend(0,0) (1, 0.5);
		\end{tikzpicture}}};}
	\hspace*{1cm}
	\tikz[remember picture]{\node(3){%
			\resizebox{!}{1.5cm}{%
				\begin{tikzpicture}
					\draw[thick] (0,1) to (0, -1);
					\draw[thick] (-0.866, 0.5) to (0.866, -0.5);
					\draw[thick] (-0.866, -0.5) to (0.866, 0.5);
					\node[above, scale = 0.7, blue] at (0,1){2};
		\end{tikzpicture}}};}
	\tikz[overlay, remember picture]{
		\draw[latex-, semithick] (2) -- (2-|3.west);}
	\hspace*{1cm}
	\tikz[remember picture]{\node(4){%
			\resizebox{!}{1.5cm}{%
				\begin{tikzpicture}
					\draw[thick] (0,0) to (2, 0);
					\draw[thick] (0.2, -0.1) to (0.2, 1);
					\draw[thick] (0.6, -0.1) to (0.6, 1);
					\draw[thick] (1, -0.1) to (1, 1);
					\node[above, scale = 0.6, blue] at (1,1){2};
					\node[below, scale = 0.6, blue] at (1.5,0) {4};
		\end{tikzpicture}}};}
	\tikz[overlay, remember picture]{
		\draw[latex-, semithick] (3) -- (3-|4.west);
	}
	\newline
	\tikz[remember picture]{\node(3){%
			\resizebox{!}{1.5cm}{%
				\begin{tikzpicture}
					\draw[thick] (0,1) to (0, -1);
					\draw[thick] (-0.866, 0.5) to (0.866, -0.5);
					\draw[thick] (-0.866, -0.5) to (0.866, 0.5);
		\end{tikzpicture}}};}
	\hspace*{1cm}
	\tikz[remember picture]{\node(4){%
			\resizebox{!}{1.5cm}{%
				\begin{tikzpicture}
					\draw[thick] (0,0) to (2, 0);
					\draw[thick] (0.2, -0.1) to (0.2, 1);
					\draw[thick] (0.6, -0.1) to (0.6, 1);
					\draw[thick] (1, -0.1) to (1, 1);
					\node[below, scale = 0.6, blue] at (1.5,0) {3};
		\end{tikzpicture}}};}
	\tikz[overlay, remember picture]{
		\draw[latex-, semithick] (3) -- (3-|4.west);
	}
	\end{figure}
	
It's easy to see that we get the same blowup sequence for multiples of types $\mathrm{II}, \mathrm{III}$ and $\mathrm{IV}$, only with all multiplicities multiplied by $m$. This yields the entries in the table above.
\end{proof}
\subsection{Genus 2}
There is a classification of special fibres of minimal regular models for curves of genus 2 by Namikawa-Ueno \cite{namikawa-ueno}. Their proof however is over $\mathbb{C}$ and uses analytic methods.

There is also a classification of all possible normal crossings reduction types by Dokchitser \cite{tim_classification}, which is valid over an arbitrary discrete valuation ring. Using the latter classification and our results on types, we can show that the types appearing in the Namikawa-Ueno classification exhaust all possibilites over an arbitrary discrete valuation ring (even in positive or mixed characteristic), and can find the correspondence between minimal regular and minimal normal crossings reduction types.
\begin{theorem}\label{genus_2_classification}
	Let $\mathcal{O}_K$ be a discrete valuation ring with fraction field $K$ and algebraically closed residue field $k$. Suppose $C/K$ is a smooth projective curve of genus 2. Then the minimal regular type of $C$ is one of the types listed in \cite{namikawa-ueno}. Its minimal normal crossings type is one of the types listed in \cite[Table G2]{tim_classification}. The correspondence between the two types is as given in Table G2 of loc.cit, or as in the tables of \cite{genus_2_tables}.
\end{theorem}
\begin{proof}
	The minimal normal crossings type of $C$ must be one of the types in \cite[Table G2]{tim_classification}, see Example 10.8 in loc.cit. We know by Theorem \ref{min_reg<->rnc} that the minimal normal crossings type determines the minimal regular type, and we can use Algorithm \ref{algorithm_blowup} to explicitly check that each normal crossings type in the table cited corresponds to the minimal regular (Namikawa-Ueno) type listed. 
	
	Alternatively one can check the correctness of the correspondence by finding examples of curves, one of each Namikawa-Ueno type, compute their minimal normal crossings types and check that they are as given in the tables cited. This is done (for the example families over $\mathbb{C}[[t]]$ listed by Namikawa-Ueno) in \cite[Proposition 2.6, Corollary 2.7]{genus_2_tables}.
\end{proof}

\printbibliography
\end{document}